\documentclass[12pt]{article}
\usepackage{amsmath,amssymb,amsfonts,amsthm,graphicx}
\usepackage[usenames,dvipsnames]{color}
\newcommand{\E}{{\mathbb E}}

\newcommand{\F}{{\mathbb F}}

\newcommand{\be}{\begin{equation}}
\newcommand{\ee}{\end{equation}}
\newcommand{\one}{\text{\textup{\texttt{1}}}}
\newcommand{\old}[1]{}
\newcommand{\eps}{\varepsilon}

\newcommand{\yellow}[1]{}

\renewcommand{\P}{{\cal P}}

\newcommand{\R}{{\mathbb R}}

\newcommand{\hh}{\hspace{.03cm}}
\newcommand{\dd}{\mbox{d}}

\newtheorem{theorem}{Theorem}

\newtheorem{lemma}[theorem]{Lemma}
\newtheorem{proposition}[theorem]{Proposition}
\newtheorem{corollary}[theorem]{Corollary}

\begin{document}
\date{}

\title
{Permutations with fixed pattern densities}
\author{
Richard Kenyon\thanks{Department of Mathematics, Brown University, Providence, RI 02912. E-mail: {\tt rkenyon at math.brown.edu}}
\and Daniel Kr\'{a}l'\thanks{Mathematics Institute, DIMAP and
  Department of Computer Science, University of Warwick, Coventry CV4
  7AL, UK. E-mail: {\tt d.kral@warwick.ac.uk}}
\and Charles Radin\thanks{Department of Mathematics, University of Texas, Austin, TX 78712. E-mail: {\tt radin@math.utexas.edu}}
\and Peter Winkler\thanks{Department of Mathematics, Dartmouth
  College, Hanover, New Hampshire 03755. E-mail: {\tt Peter.Winkler@Dartmouth.edu}}
}

\maketitle

\begin{abstract}
We study scaling limits of random permutations (``permutons'')
constrained by having fixed densities of a finite number of patterns.
We show that the limit shapes are determined
by maximizing entropy over permutons with those constraints.
In particular, we compute (exactly or numerically)
the limit shapes with fixed $1\hh 2$ density,
with fixed $1\hh 2$ and \hbox{1\hh2\hh3} densities,
with fixed \hbox{1\hh2} density and the sum of \hbox{1\hh2\hh3} and \hbox{2\hh1\hh3} densities, and
with fixed \hbox{1\hh2\hh3} and \hbox{3\hh2\hh1} densities.
In the last case we explore a particular phase transition.
To obtain our results, we also provide a description of permutons using a dynamic construction.
\end{abstract}

\section{Introduction}
{We study pattern densities in permutations. 
A \emph{pattern} $\tau\in S_k$ in  a permutation $\sigma\in S_n$ (with $k\le n$) is a
$k$-element subset of indices
$1\le i_1<\dots<i_k\le n$ whose image under $\sigma$ has the same order as that under $\tau$. 
For example the first three indices in the permutation $4312$ have pattern $3\hh2\hh1$.
The \emph{density} of $\tau\in S_k$ in $\sigma\in S_n$ is $\binom{n}{k}^{-1}$ times the number of such subsets of indices.
 
Pattern \emph{avoidance} in permutations is a well-studied and rich area of combinatorics;
see \cite{Ki} for the history and the current state of the subject.
Less studied is the problem of determining the range of possible densities of patterns, and
the ``typical shape'' of permutations with constrained densities of (a fixed set of) patterns.
We undertake such a study here.
Specifically, we consider the densities of one or more patterns and consider the
\emph{feasible region}, or {phase} space $\F$, of possible values of
densities of the chosen patterns for permutations in $S_n$ in the limit of large $n$.
For densities in the interior of $\F$ we
study the shape of a typical permutation with those densities, again in the large $n$ limit.
We note that
the typical shape of pattern-avoiding permutations 
(which necessarily lie on the boundary of the feasible region $\F$)
has also recently been investigated~\cite{AM, DP, HRS, ML, MP1, MP2}.

To deal with these asymptotic
questions we show that the size of our target sets of constrained
permutations can be
estimated by maximizing a certain function over limit objects called
permutons.  Furthermore when---as appears to be usually the
case---the maximizing permuton is unique, properties of most
permutations in the class can then be deduced from it. After
setting up our general framework we work out several examples. 
To give further details we need some notation.
}

To a permutation $\pi\in S_n$ one can associate a
probability measure $\gamma_\pi$ on $[0,1]^2$ as
follows.  Divide $[0,1]^2$ into an $n\times n$ grid of squares of size
$1/n\times1/n$. Define the density of $\gamma_\pi$ on the square in
the $i$th row and $j$th column to be the constant $n$ if $\pi(i)=j$
and $0$ otherwise.  In other words, $\gamma_\pi$ is a geometric representation
of the permutation matrix of $\pi$.

Define a \emph{permuton} to be a probability measure $\gamma$ on $[0,1]^2$ with
uniform marginals:
\be \gamma([a,b] \times [0,1])=b-a=\gamma([0,1] \times [a,b]),\ \
\hbox{for all }0\le
a\le b\le 1.\ee
Note that $\gamma_\pi$ is a permuton for any permutation $\pi \in S_n$.
Permutons were introduced in \cite{HKMS1,HKMS2} with a different but equivalent definition;
the measure theoretic view of large permutations can be traced to
\cite{PS} and was used in \cite{GGKK,KP} as an analytic representation of
permutation limits equivalent to that used in \cite{HKMS1,HKMS2};
the term ``permuton'' first appeared, we believe, in \cite{GGKK}.

Let $\Gamma$ be the space of permutons.  {There is a natural topology on $\Gamma$,
the weak topology on probability measures, which can equivalently be defined as the metric topology}
defined by the metric $d_\square$ given
by $d_\square(\gamma_1,\gamma_2) = \max |\gamma_1(R) - \gamma_2(R)|$, where $R$
ranges over aligned rectangles in $[0,1]^2$.  This topology is also the same as
that given by the $L^\infty$ metric on the cumulative distribution
functions $G_i(x,y) = \gamma_i([0,x] \times [0,y])$. 
We say that a sequence of permutations $\pi_n$ with $\pi_n\in S_n$ \emph{converges} as $n\to\infty$ 
if the associated permutons converge in the above sense.

{Extending the definition above, given a permuton $\gamma$ the \emph{pattern density of $\tau$
in $\gamma$}}, denoted $\rho_\tau(\gamma)$, is by definition the probability that,
when $k$ points are selected independently from $\gamma$ and their $x$-coordinates are ordered,
the permutation induced by their $y$-coordinates is $\tau$.
For example, for $\gamma$ with probability density $g(x,y)\dd{}x\,\dd{}y$,
the density of pattern $1\hh2\in S_2$ in $\gamma$ is
\be\label{2}
\rho_{12}(\gamma)=2\int_{x_1<x_2\in[0,1]}\int_{y_1<y_2\in[0,1]}
g(x_1,y_1)g(x_2,y_2)\dd{}x_1\dd{}y_1\dd{}x_2\dd{}y_2.\ee 

It follows from results of \cite{HKMS1,HKMS2} that two permutons are equal if they
have the same pattern densities (for all $k$).

The notion of pattern density for permutons generalizes the notion for permutations.
Note however that the density of a pattern $\alpha\in S_k$ in a permutation $\tau\in S_n$ 
(defined to be the number of copies
of $\alpha$ in $\tau$, divided by $\binom{n}{k}$) will not generally
be equal to the density of $\alpha$ in the permuton $\gamma_\tau$; equality will only hold in the limit of large $n$.

\subsection{Results}

Theorem \ref{thm:tras} below (restated from the somewhat different form in Trashorras, \cite{Tr}) 
is a large deviations theorem for permutons: it describes explicitly how
many large permutations lie near a given permuton. 
The statement is essentially that the number of permutations in $S_n$ lying near a permuton $\gamma$ is
\be n!e^{(H(\gamma)+o(1))n},\ee
where $H(\gamma)$ is the ``permuton entropy'' (defined below). 

We use this large deviations theorem to prove Theorem \ref{thm:constrained}, which 
describes both the number and (when uniqueness holds)
limit shape of permutations in which a finite number of pattern densities
have been fixed. The theorem is a variational principle: it shows that the number of such permutations is determined
by the permuton entropy {\it maximized} over the set of permuton(s) having those fixed pattern densities.

Another construction we use
replaces permutons by families of \emph{insertion measures} $\{\mu_t\}_{t\in[0,1]}$,
which is analogous to building a permutation by inductively inserting one element at a time into a growing
list: for each $i\in[n]$ one inserts $i$ into a random location in the permuted list of the first $i-1$ elements.
This construction is used to describe explicitly the entropy maximizing permutons with fixed densities
of patterns of type $\ast\!\ast\dots\ast i$ (here each $\ast$ represents an element not exceeding the length of the pattern,
for example,
$\ast\ast 2$ represents the union of the patterns $1\hh3\hh2$ and $3\hh1\hh2$). We prove that for this family of patterns
the maximizing permutons are analytic, the entropy function as a function of the constraints is
analytic and strictly concave, and the optimal permutons are unique and have analytic probability densities.

The most basic example to which we apply our results, the
entropy-maximizing permuton for a fixed density $\rho_{12}$ of $12$
patterns, has probability density
$$g(x,y) = \frac{r(1-e^{-r})}{(e^{r(1-x-y)/2}-e^{r(x-y-1)/2}-e^{r(y-x-1)/2}+e^{r(x+y-1)/2})^2}$$
where $r$ is an explicit function of $\rho_{12}$.
See Figure \ref{12permuton}.
\begin{figure}[htbp]
\center{\includegraphics[width=1.8in]{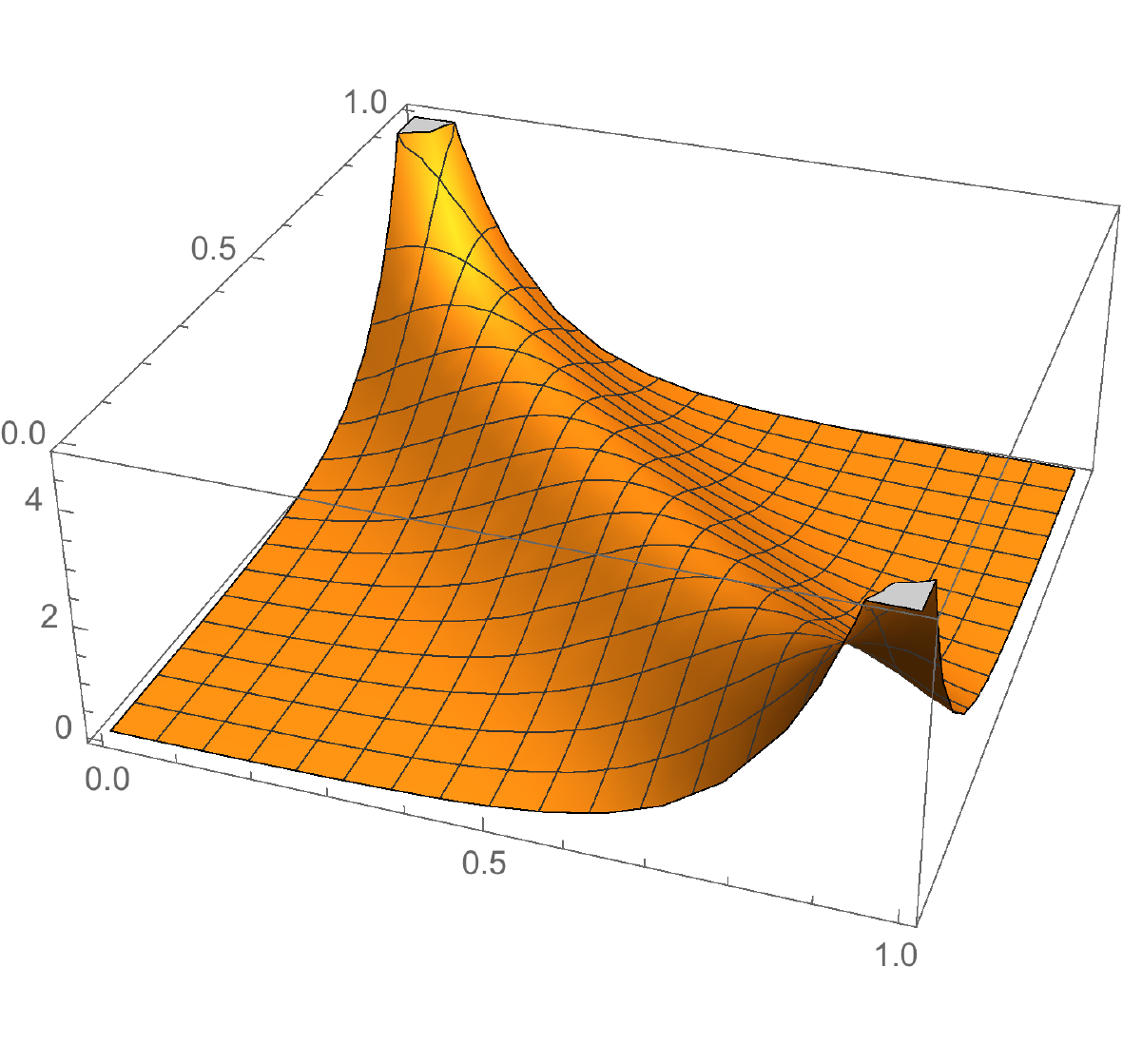}\includegraphics[width=1.8in]{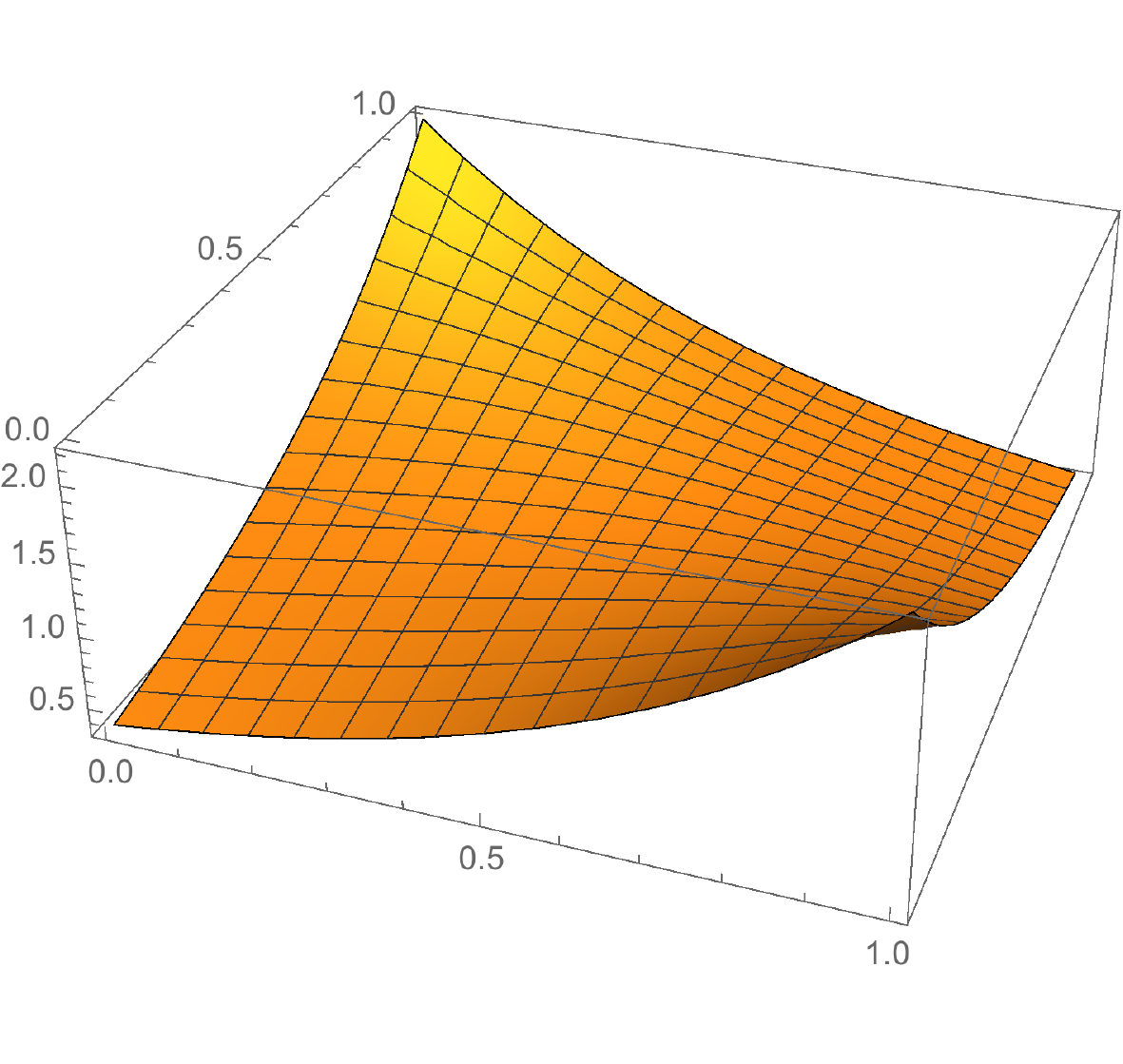}\includegraphics[width=1.8in]{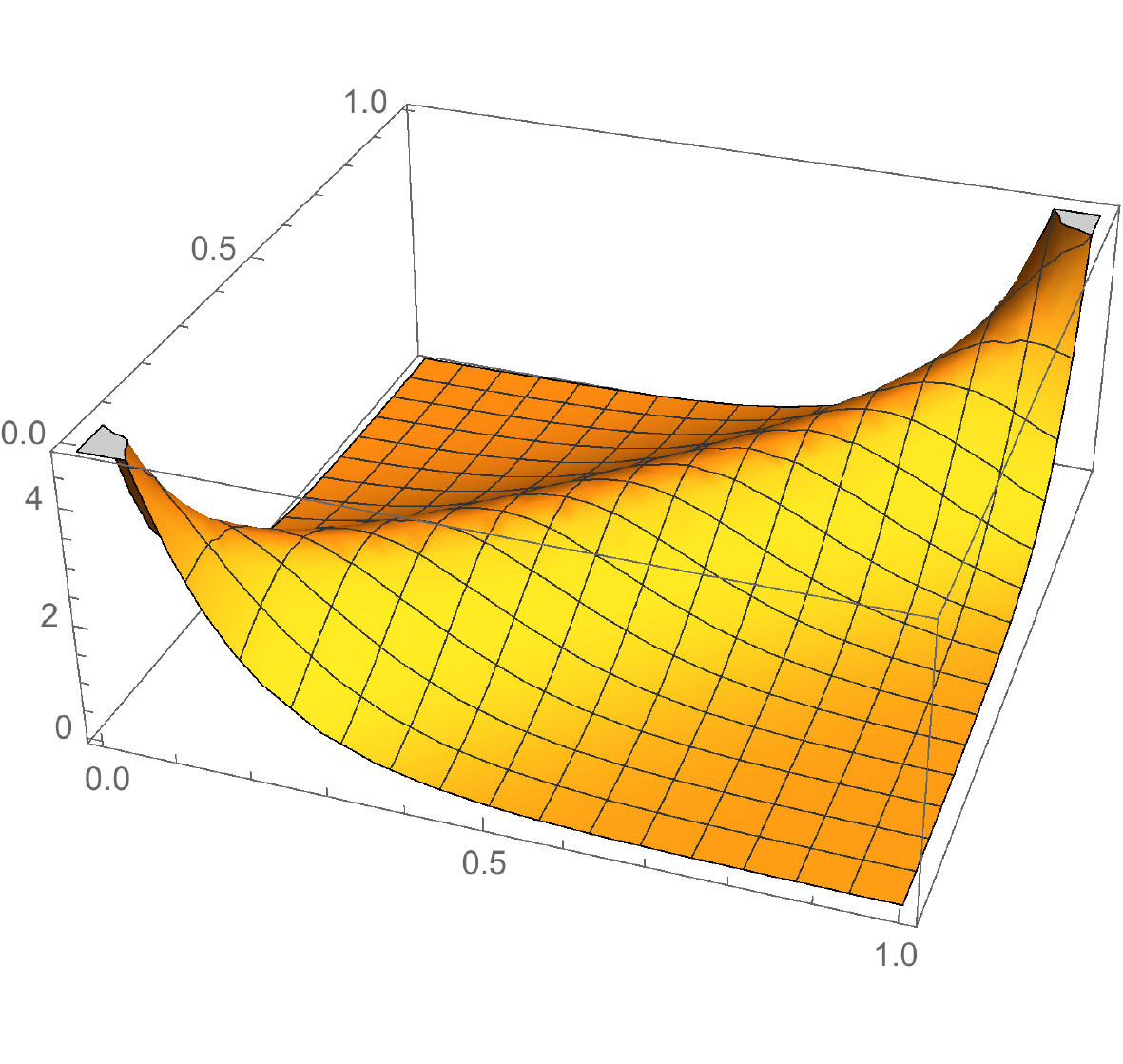}}
\caption{\label{12permuton}The permuton with fixed density $\rho$ of pattern $12$, shown for $\rho=.2,.4,.8$.}
\end{figure}

While maximizing permutons can be shown to satisfy certain explicit PDEs (see Section \ref{PDEsection}),
they can also exhibit a very diverse set of behaviors. Even in one of the simplest cases, that of fixed density
of the two patterns $12$ and $1\hh2\hh3$, the variety of shapes of permutons (and therefore of the approximating
permutations) is remarkable: see Figure \ref{scallopedtri}. In this case we prove that
the feasible region of densities is the so-called ``scalloped
triangle'' of Razborov \cite{R, R2} 
which also describes the space
of feasible densities for edges and triangles in the graphon
model.

Another example which has been studied recently \cite{EN, HLNPS1, HLNPS2} is the case of the two patterns $1\hh2\hh3$ and $3\hh2\hh1$.
In this case we describe a phase transition in the feasible region, where the maximizing
permuton changes abruptly. 

The variational principle can easily be extended to analyze other
constraints that are continuous in the permuton topology. For constraints that are not continuous,
for example the number of cycles of a fixed size, one can analyze an analogous
``weak'' characteristic, which is continuous, by applying the characteristic to patterns.
For example, while the number of fixed points of a permuton is not well-defined,
we can compute the expected number of fixed points for the permutation in $S_n$ obtained
by choosing $n$ points independently from the permuton, and analyze this quantity in the large $n$ limit.
This computation will be discussed in a subsequent paper \cite{KKRW2};
the result is that the expected weak number of fixed points is $$\int_0^1 g(x,x)\,\dd{}x$$ when $g$ has a continuous density.
Similar expressions hold for cycles of other lengths.

\subsection{Analogies with graphons}

For those who are familiar with variational principles for dense
graphs \cite{CV, CD,  RS1, RS2}, we note the following differences between the graph case and
the permutation case (see \cite{L} for background on graph asymptotics):

\begin{enumerate}
\item Although permutons serve the same purpose for permutations that graphons serve for graphs,
and (being defined on $[0,1]^2$) are superficially similar, they are measures (not symmetric functions)
and represent permutations in a different way.  (One {\em can} associate a graphon with a limit
of permutations, via comparability graphs of two-dimensional posets, but these have trivial entropy in
the Chatterjee-Varadhan sense \cite{CV} and we do not consider them here.)

\item The classes of constrained (dense) graphs considered in \cite{CV} have size about $e^{c n^2}$, $n$ being the
number of vertices and the (nonnegative) constant $c$ being the target of study.  Classes of
permutations in $S_n$ are of course of size at most $n! \sim e^{n(\log
  n - 1)}$ but the constrained ones we consider
here have size of order not $e^{cn \log n}$ for $c \in (0,1)$, as one might at first expect, but
instead $e^{n \log n - n + cn}$ where $c \in [-\infty,0]$ is the
quantity of interest.

\item The ``entropy'' function, i.e., the function of the limit structure to be maximized, is bounded for
graphons but unbounded for permutons.  This complicates the analysis for permutations.

\item The limit structures that maximize the entropy function tend, in the graph case, to be
combinatorial objects: step-graphons corresponding to what Radin, Ren
and Sadun call ``multipodal'' graphs \cite{RRS}.
In contrast, maximizing permutons at interior points of feasible regions seem always to be smooth
measures with analytic densities.  Although they are more complicated than maximizing graphons,
these limit objects are more suitable for classical variational analysis, e.g., differential
equations of the Euler-Lagrange type.
\end{enumerate}

\section{Variational principle}

For convenience, we denote the unit square $[0,1]^2$ by $Q$.

Let $\gamma$ be a permuton with density $g$ defined almost everywhere.  We compute the {\em permutation entropy} $H(\gamma)$ of $\gamma$ as follows:
\be
H(\gamma) = \int_Q -g(x,y) \log g(x,y)\,\dd{}x\, \dd{}y
\ee
where ``$0 \log 0$'' is taken as zero.  Then $H$ is finite whenever $g$
is bounded (and sometimes when it is not).
In particular for any $\sigma \in S_n$, we have $H(\gamma_\sigma)=-n(n \log{n} /n^2) = -\log n$ and 
therefore $H(\gamma_\sigma) \to -\infty$
for any sequence of increasingly large permutations even though $H(\lim \gamma_\sigma)$ may be finite.
Note that $H$ is zero on the uniform permuton (where $g(x,y) \equiv 1$) and negative (sometimes $-\infty$) on
all other permutons, since the function $z \log z$ is concave downward.  If $\gamma$ has no density, we define $H(\gamma)
= -\infty$.

We use the following large deviations principle, first stated in a somewhat different form by Trashorras
(Theorem 1 in \cite{Tr}); see also Theorem 4.1 in \cite{Mu}. In Section \ref{appendix}
we give an alternative proof.

\begin{theorem}[\cite{Tr}]\label{thm:tras} Let $\Lambda$ be a set of permutons, $\Lambda_n$ the set of permutations $\pi \in S_n$ with
$\gamma_\pi \in \Lambda$.  Then:
\begin{enumerate}
\item If $\Lambda$ is closed,
\be
\lim_{n \to \infty} \frac1n \log \frac{|\Lambda_n|}{n!} \le \sup_{\gamma \in \Lambda} H(\gamma);
\ee
\item If $\Lambda$ is open,
\be
\lim_{n \to \infty} \frac1n \log \frac{|\Lambda_n|}{n!} \ge \sup_{\gamma \in \Lambda} H(\gamma).
\ee
\end{enumerate}
\end{theorem}

To make a connection with our applications to large
constrained permutations, fix some finite set $\P = \{\pi_1,\dots,\pi_k\}$ of patterns.
Let $\alpha = (\alpha_1,\dots,\alpha_k)$ be a vector of desired pattern densities. We then define
two sets of permutons:
\be
\Lambda^{\alpha,\varepsilon}=\{\gamma\in \Gamma\,|\, |\rho_{\pi_j}(\gamma) - \alpha_j| < \varepsilon~\mbox{for each}~1 \le j \le k\}
\ee
and 
\be
\Lambda^{\alpha}=\{\gamma\in \Gamma\,|\, \rho_{\pi_j}(\gamma) = \alpha_j~\mbox{for each}~1 \le j \le k\}.
\ee

With that notation, and the understanding that
$\Lambda_n^{\alpha,\varepsilon} = \Lambda^{\alpha,\varepsilon} \cap
\gamma(S_n)$, where $\gamma(\alpha)=\gamma_\alpha$ as before, our first main
result is:

\begin{theorem}\label{thm:constrained}
$$\lim_{\varepsilon \downarrow 0}\lim_{n\to\infty}\frac1n \log \frac{|\Lambda_n^{\alpha,\varepsilon}|}{n!} = \max_{\gamma \in \Lambda^{\alpha}} H(\gamma).$$
\end{theorem}

The value $\max_{\gamma \in \Lambda^{\alpha}} H(\gamma)$ (which is guaranteed by the theorem to exist, but may be $-\infty$)
will be called the {\em constrained entropy} and denoted by
$s(\alpha)$. In Section~\ref{sect:proof-constrained} we will prove Theorem~\ref{thm:constrained}.

Theorem~\ref{thm:constrained} puts us in a position to try to describe and enumerate permutations with some given pattern densities.  It does not, of course,
guarantee that there is just one $\gamma \in \Lambda^{\alpha}$ that maximizes $H(\gamma)$, nor that there is one with finite entropy.  As we shall see it
seems to be the case that interior points in feasible regions for pattern densities do have permutons with finite entropy, and {\em usually} just one optimizer.
Points on the boundary of a feasible region (e.g., pattern-avoiding permutations) often have only singular permutons, and since the latter always have
entropy $-\infty$, Theorem~\ref{thm:constrained} will not be of direct use there.

{
\section{Feasible regions and entropy optimizers}

We collect here some general facts about
feasible regions and entropy optimizers, making use of concavity of entropy and the ``heat flow on permutons''.

\subsection{Heat flow on permutons}
The heat flow is 
a continuous flow on the space of permutons
with the property that for any permuton $\mu=\mu_0$ and any positive time $t>0$,
$\mu_t$ has analytic density (and thus finite entropy).

The flow after time $t$ is given by the action of the heat operator $e^{t\Delta}$ where $\Delta$ is the 
Laplacian on the square with reflecting boundary conditions.
One can describe the flow concretely as follows.
First, one can describe a permuton $\mu$ by its characteristic function $\hat g(u,v) = \E[e^{i(ux+vy)}]$.
In fact since we are on the unit square we can use instead the discrete Fourier cosine series
$$\hat g(j,k) = \E[\cos(\pi j x)\cos(\pi k y)]$$
with $j,k\ge 0$.

The operator $e^{t\Delta}$ acts on the coefficients by multiplication by $e^{-(j^2+k^2)t}$:
$$\hat g_t(j,k) = \hat g_0(j,k)e^{-(j^2+k^2)t}.$$ Note that the heat flow
preserves the marginals, that is $\hat g_t(j,0) = \delta_j = \hat g_0(j,0)$ and $\hat g_t(0,k)=\delta_k=\hat g_0(0,k)$.

For any $t>0$ the Fourier coefficients $\hat g_t(j,k)$ then decay exponentially quickly so that $\hat g_t(j,k)$ are the
Fourier coefficients of a measure with analytic density.

\subsection{Elementary consequences for feasible regions}

Let $R$ be the feasible region for permutons with some finite set of pattern densities. Let
$R_M$ be the subset of $R$ consisting of points representable by
an analytic permuton with entropy at least $-M$, and $R_*$ those representable by permutons with finite
entropy.

Entropy is upper-semicontinuous on $R$ (just as it is on the
space $\Gamma$ of all permutons). So $R_M$ is closed.

\begin{theorem}\label{thm:Dense} $R_*$ is dense in $R$.
\end{theorem}

\begin{proof}
Any permuton $\gamma$ may be perturbed (by the heat flow for small time, thus moving densities only a small amount)
to achieve an analytic permuton with finite entropy.
\end{proof}

\begin{theorem}\label{thm:interior}
Let $C$ be a topological sphere 
in $R_M$.  Then the interior of $C$ is contained in $R_M$.
\end{theorem}

\begin{proof}
Consider the space of all permutons obtainable as convex combinations
of the entropy-maximizing permutons on $C$. This set is convex and contains a topological disk whose boundary is the
set of entropy-maximizing permutons on $C$, parameterized in order. The image of this disk
in $R$ provides a homotopy of $C$ to a point and thus contains the interior of $C$.
Since the entropy function is concave, the entropies of the points
in the space of convex combinations are all at least $-M$.
\end{proof}

We now give several corollaries of Theorem~\ref{thm:interior}.

\begin{corollary}\label{cor:extrema}
$R$ contains no local minimum of the entropy, nor any local maximum that is not a
global maximum of the entropy.
\end{corollary}

\begin{proof}
A local minimum is the minimum in a disk around it. Let $C$ as in the proof of Theorem \ref{thm:interior}
be the boundary of this disk. Concavity of entropy
implies that the minimum of the entropy on the disk occurs on $C$, a contradiction.

For the second statement, convex combinations connect two local maxima with different values by a path
in $R$ whose entropies are lower bounded by the minimum of the two values, a contradiction.
\end{proof}

\begin{corollary}\label{cor:1-dim}
If $R$ is the feasible region for a single density, then
it is an interval on the interior of which entropy is finite and concave.
\end{corollary}

Note that for any pattern $\pi$ 
there is a permuton which has zero density for that pattern (either the identity permuton or the `anti-identity' permuton).
The \emph{maximal $\pi$-density permuton(s)}
are not known in general, although a lower bound on the maximal density is obtained from the permuton  $\gamma_\pi$.

\begin{corollary}\label{cor:plane}
If $R$ is the feasible region for two densities, then $R_M$ is
simply connected and $R$ and $R_*$ are connected.
\end{corollary}
}

\section{Proof of Theorem~\ref{thm:constrained}}
\label{sect:proof-constrained}

Since we will be approximating $H$ by Riemann sums, it is useful to define, for any permuton $\gamma$
and any positive integer $m$, an approximating ``step-permuton'' $\gamma^m$ as follows.  Fixing $m$,
denote by $Q_{ij}$ the half-open square $((i{-}1)/m,i/m] \times ((j{-}1)/m,j/m]$; for each $1 \le i,j \le m$,
we want $\gamma^m$ to be uniform on $Q_{ij}$ with $\gamma^m(Q_{ij}) = \gamma(Q_{ij})$.  In terms of
the density $g^m$ of $\gamma^m$, we have $g^m(x,y) = m^2 \gamma(Q_{ij})$ for all $(x,y) \in Q_{ij}$.

To prove Theorem~\ref{thm:constrained} we use the following result.

\begin{proposition}\label{thm:converge}
For any permuton $\gamma$, $\lim_{m \to \infty} H(\gamma^m) = H(\gamma)$,
with $H(\gamma^m)$ diverging downward when $H(\gamma) = -\infty$.
\end{proposition}

In what follows we will, in order to increase readability, write
\be
\int_0^1 \int_0^1  -g(x,y) \log g(x,y) \dd{}x \dd{}y
\ee
as just $\int_Q -g \log g$.  Also for the sake of readability, we will for this section only state results in terms
of $g \log g$ rather than $-g \log g$; this avoids clutter caused by
a multitude of absolute values and negations.  Eventually, however, we will need to deal with
an entropy function $H(\gamma) = \int_Q -g \log g$ that takes values in $[-\infty,0]$.

Define
\be
g_{ij} = m^2\gamma(Q_{ij}).
\ee
We wish to show that the Riemann sum
\be
\frac{1}{m^2}\sum_{0 \le i,j \le m} g_{ij} \log g_{ij},
\ee
which we denote by $R_m(\gamma)$, approaches $\int_Q g \log g$ when $\gamma$ is absolutely continuous with respect to Lebesgue measure,
i.e., when the density $g$ exists a.e., and otherwise diverges to $\infty$. There are
thus three cases:
\begin{enumerate}
\item $g$ exists and $\int_Q g \log g < \infty$;
\item $g$ exists but $g \log g$ is not integrable, i.e., its integral is $\infty$;
\item $\gamma$ is singular.
\end{enumerate}

Let $A(t) = \{(x,y)\in Q:~g(x,y) \log g(x,y) > t\}$.

In the first case, we have that $\lim\sup \int_{A(t)} g \log g = 0$,  and since $g\log g\ge t$ on $A(t)$, we have
$\lim\sup |A(t)|t = 0$ where $|A|$ denotes
the Lebesgue measure of $A \subset Q$.  (We
need not concern ourselves with large negative values, since the function $x \log x$ is bounded below by $-1/e$.)

In the second case, we have the opposite, i.e., for some $\eps>0$ and any $s$ there is a $t > s$
with $t|A(t)| > \eps$.

In the third case, we have a set $A \subset Q$ with $\gamma(A)>0$ but $|A|=0$.

In the proof that follows we do not use the fact that $\gamma$ has uniform marginals, or that it is normalized to have $\gamma(Q)=1$.
Thus we restate Proposition~\ref{thm:converge} in greater generality:

\begin{proposition}\label{thm:limits}
Let $\gamma$ be a finite measure on $Q = [0,1]^2$ and $R_m=R_m(\gamma)$. Then:
\begin{enumerate}
\item If $\gamma$ is absolutely continuous with density $g$, and $g \log g$ is integrable, then $\lim_{m \to \infty} R_m = \int_Q g \log g$.
\item If $\gamma$ is absolutely continuous with density $g$, and $g \log g$ is not integrable, then $\lim_{m \to \infty} R_m = \infty$.
\item If $\gamma$ is singular, then $\lim_{m \to \infty} R_m = \infty$.
\end{enumerate}
\end{proposition}
\begin{proof}
We begin with the first case, where we need to show that for any $\eps > 0$,
there is an $m_0$ such that for $m \ge m_0$,
\be\label{eq:*}
\int_Q g \log g - \frac{1}{m^2}\sum_{i,j=0}^m g_{ij} \log g_{ij} < \eps~.
\ee
Note that since $x \log x$ is convex, the quantity on the left cannot be negative.

\begin{lemma}\label{lemma:s}
Let $\eps > 0$ be fixed and small.  Then there are $\delta>0$ and $s$ with the following properties:
\begin{enumerate}
\item $\lvert\int_{A(s)} g \log g\rvert < \delta^2/4$;
\item $|A(s)| < \delta^2/4$;  
\item for any $u,v \in [0,s+1]$, if $|u-v| < \delta$ then $|u \log u - v \log v| < \eps/4$;
\item for any $B \subset Q$, if $|B|< 2\delta$ then $\int_B |g \log g| < \eps/4$.
\end{enumerate}
\end{lemma}

\begin{proof}
By Lebesgue integrability of $g \log g$, we can immediately choose $\delta_0$ such that any $\delta<\delta_0$ will satisfy the fourth property.

We now choose $s_1$ so that $\int_{A(s_1)} g \log g < \delta_0^2/4$, and $t|A(t)| < 1$ for all $t \ge s_1$.
Since $[0,s_1]$ is compact we may choose $\delta_1 < \delta_0$ such that for any $u,v \in [0,s_1+1]$, $|u-v|<\delta_1$
implies $|u \log u - v \log v| < \eps/4$.  We are done if $|A(s_1)| < \delta_1^2/4$ but since $\delta_1$ depends on
$s_1$, it is not immediately clear that we can achieve this.  However, we know that since $\frac{d}{du} u \log u = 1 + \log u$,
the dependence of $\delta_1$ on $s_1$ is only logarithmic, while $|A(s_1)|$ descends at least as fast as $1/s_1$.  

So we take $k = \lceil \log (\delta/2) / \log (\eps/2) \rceil$ and let $\delta = \delta_1/k$, $s = s_1^k$.
Then $u,v \in [0,s+1]$ and $|u-v|<\delta$ implies $|u \log u - v \log v| < \eps/4$, and
\be
\int_{A(s)}g \log g \le \left(\int_{A(s_1)}g \log g\right)^k < (\eps/2)^{2\log (\delta/2) / \log (\eps/2)} = \delta^2/4
\ee
as desired.  Since $u \log u > u > 1$ for $u>e$, we get $|A(s)| < \delta^2/4$. 
\end{proof}

Henceforth $s$ and $\delta$ will be fixed, satisfying the conditions of Lemma~\ref{lemma:s}.  
Since $g$ is measurable we can find a subset $C \subset Q$ with $|C| = |A(s)| < \delta^2/4$ such that
$g$, and thus also $g \log g$, is continuous on $Q \setminus C$.  Since $\int_B g \log g$ is maximized
by $B=A(s)$ for sets $B$ with $B=|A(s)|$, $|\int_C g \log g| < \int_{A(s)} g \log g$,
so $|\int_{A(s) \cup C} g \log g| < \delta^2/2$.  We can then find an open set $A$ containing $A(s) \cup C$
with $|A|$ and $\int_A g \log g$ both bounded by $\delta^2$.

We now invoke the Tietze Extension Theorem to choose a continuous $f:~Q \to {\mathbb R}$ with
$f(x,y) = g(x,y)$ on $Q \setminus A$, and $f \log f < s$ on all of $Q$.
Since $f$ is continuous and bounded, $f$ and $f \log f$ are Riemann integrable.
Let $f_{ij}$ be the mean value of $f$ over $Q_{ij}$, i.e.,
\be
f_{ij} = m^2 \int_{Q_{ij}}f~.
\ee
Since, on any $Q_{ij}$, $\inf f \log f \le f_{ij} \log f_{ij} < \sup f \log f$,
we can choose $m_0$ such that $m \ge m_0$ implies
\be\label{eq:1}
\left|\int_Q f \log f - \frac{1}{m^2}\sum_{ij} f_{ij} \log f_{ij}\right| < \eps/4~.
\ee
We already have
\begin{eqnarray}\label{eq:2} |\int_Q g \log g&-&\int_Q f \log f | = |\int_A g \log g - \int_A f \log f |\cr
&\le& |\int_A g \log g| - |\int_A f \log f| < 2\delta^2 \ll \eps/4.
\end{eqnarray}
Thus, to get (\ref{eq:*}) from (\ref{eq:1}) and (\ref{eq:2}), it suffices to bound 
\be \left|\frac{1}{m^2}\sum_{ij} g_{ij} \log g_{ij} - \frac{1}{m^2}\sum_{ij} f_{ij} \log f_{ij}\right| \ee
by $\varepsilon/2$.

Fixing $m > m_0$, call the pair $(i,j)$, and its corresponding square $Q_{ij}$, ``good'' if $|Q_{ij}\cap A| < \delta/(2m^2)$.
The number of bad (i.e., non-good) squares cannot exceed $2 \delta m^2$, else $|A| > 2\delta m^2 \delta/(2m^2) = \delta^2$.

For the good squares, we have
\be
\left|g_{ij} - f_{ij}\right| = m^2\left|\int_{Q_{ij} \cap A} (g - f) \right| \le m^2\left|\int_{Q_{ij} \cap A} 2g\right| \le 2(\delta/2) = \delta
\ee
with $f_{ij} \le s$, thus $f_{ij}$ and $g_{ij}$ both in $[0,s+1]$.  It follows that
\be \left|g_{ij}\log g_{ij} - f_{ij}\log f_{ij}\right| < \eps/4\ee
 and therefore the ``good'' part of the Riemann sum discrepancy, namely
\be
\frac{1}{m^2}\left|\sum_{{\rm good}~ij} \left( g_{ij} \log g_{ij} - f_{ij} \log f_{ij} \right)\right|~,
\ee
is itself bounded by $\eps/4$.

Let $Q'$ be the union of the bad squares, so $|Q'| < m^2 2\delta/(2m^2) = 2\delta$; then
by (\ref{eq:1}) and convexity of $u \log u$,
\be
\frac{1}{m^2}\left|\sum_{{\rm bad}~ij} g_{ij} \log g_{ij} - f_{ij} \log f_{ij}\right| < 2\left|\int_{Q'} g \log g\right| < 2(\eps/8) = \eps/4
\ee
and we are done with the case where $g \log g$ is integrable.

\medskip

Suppose $g$ exists but $g \log g$ is not integrable; we wish to show that for any $M$, there is an $m_1$ such that $m \ge m_1$ implies
$\frac{1}{m^2}\sum g_{ij} \log g_{ij} > M$.

For $t \ge 1$, define the function $g^t$ by $g^t(x,y) = g(x,y)$ when $g(x,y) \le t$, i.e.\ when $(x,y) \not\in A(t)$, otherwise $g^t(x,y) = 0$.
Then $\int_Q g^t \log g^t \to \infty$ as $t \to \infty$, so we may take $t$ so that $\int_Q g^t \log g^t \ge M+1$.  Let $\gamma^t$ be
the (finite) measure on $Q$ for which $g^t$ is the density.  Since $g^t$ is bounded (by $t$), $g^t \log g^t$ is integrable and we may apply the
first part of Proposition~\ref{thm:limits} to get an $m_1$ so that $m \ge m_1$ implies that $R^m(\gamma^t) > M$.  

Since $t \ge 1$, $g\log g \ge g^t \log g^t$ everywhere and hence, for every $m$, $R_m(\gamma^t) \le R_m(\gamma)$.  It follows that
$R_m(\gamma) > M$ for $m \ge m_1$ and this case is done.

\medskip

Finally, suppose $\gamma$ is singular and let $A$ be a set of Lebesgue measure zero for which $\gamma(A) = a > 0$.

\begin{lemma}\label{lemma:cover}  For any $\eps>0$ there is an $m_2$ such that $m>m_2$ implies that there are $\eps m^2$ squares of the $m \times m$ grid that
cover at least half the $\gamma$-measure of $A$.
\end{lemma}

\begin{proof}
Note first that if $B$ is an open disk in $Q$ of radius at most $\delta$, then for $m > 1/(2\delta)$, then we can cover
$B$ with cells of an $m \times m$ grid of total area at most $64 \delta^2$.  The reason is that such a disk cannot
contain more than $\lceil 2 \delta / (1/m) \rceil ^2 < (4 \delta m)^2$ grid vertices, each of which can be a corner
of at most four cells that intersect the disk.  Thus, rather conservatively, the total area of the cells that intersect
the disk is bounded by $(4/m^2) \cdot (4 \delta m)^2 = 64 \delta^2$.  It follows that as long as a disk has radius
at least $1/(2m)$, it costs at most a factor of $64/\pi$ to cover it with grid cells.

Now cover $A$ with open disks of area summing to at most $\pi \eps/64$.  Let $b_n$ be the $\gamma$-measures of the union of the disks
of radii at least $1/2n$. Choose $m_2$ such that $b_{m_2}>a/2$ to get the desired result.
\end{proof}

Let $M$ be given and use Lemma~\ref{lemma:cover} to find $m_2$ such that for any $m \ge m_2$, there is a set $I \subset \{1,\dots,m\}^2$
of size at most $\delta m^2$ such that $\gamma(\bigcup_I Q_{ij}) > a/2$, where $Q_{ij} = ((i{-}1)/m,i/m] \times ((j{-}1)/m,j/m]$
as before and $\delta$ is a small positive quantity depending on $M$ and $a$, to be specified later.  Then
\begin{eqnarray}
R_m(\gamma) & = & \sum_{ij} \frac{1}{m^2} g_{ij} \log g_{ij} \cr
& \ge & -1/e + \frac{1}{m^2} \delta m^2 {\bar g} \log {\bar g} = -1/e + \delta {\bar g} \log {\bar g}
\end{eqnarray}
where $\bar g$ is the mean value of $g_{ij}$ over $(i,j) \in I$, the last inequality following from the convexity of $u \log u$.
The $-1/e$ term is needed to account for possible negative values of $g \log g$.

But $\sum_I g_{ij} = m^2\gamma(\bigcup_I Q_{ij}) > m^2 a/2$, so ${\bar g} > (m^2 a/2)/(\delta m^2) = a/(2\delta)$.
Consequently
\be
R_m(\gamma) > -\frac{1}e + \delta \frac{a}{2\delta} \log \frac{a}{2\delta} = -\frac1e + \frac{a}{2} \log \frac{a}{2\delta} ~.
\ee
Taking
\be
\delta = \frac{a}{2} \exp\left( -2\frac{\left(M + \frac1e\right)}{a} \right)
\ee
gives $R_m(\gamma) > M$ as required, and the proof of Proposition~\ref{thm:converge} is complete.
\end{proof}

We now prove Theorem~\ref{thm:constrained}.

\begin{proof}
The set $\Lambda^{\alpha,\varepsilon}$ of permutons under consideration consists of those for which certain pattern densities
are close to values in the vector $\alpha$.  Note first that since the density function $\rho(\pi,\cdot)$ is continuous in the topology of $\Gamma$,
$\Lambda^{\alpha}$ is closed and by compactness $H(\gamma)$ takes a maximum value on $\Lambda^{\alpha}$.

Again by continuity of $\rho(\pi,\cdot)$, $\Lambda^{\alpha,\eps}$
is an open set and we have from the second statement of Theorem~\ref{thm:tras} that for any $\varepsilon$,
\be
\lim_{n\to\infty}\frac1n \log \frac{|\Lambda_n^{\alpha,\varepsilon}|}{n!} \ge \max_{\gamma \in \Lambda^{\alpha,\varepsilon}} H(\gamma) \ge \max_{\gamma \in \Lambda^{\alpha}} H(\gamma)
\ee
from which we deduce that 
\be
\lim_{\varepsilon \downarrow 0}\lim_{n\to\infty}\frac1n \log \frac{|\Lambda_n^{\alpha,\varepsilon}|}{n!} \ge \max_{\gamma \in \Lambda^{\alpha}} H(\gamma).
\ee

To get the reverse inequality, fix a $\gamma \in \Lambda^{\alpha}$ maximizing $H(\gamma)$.  Let $\delta>0$; since $H$ is upper semi-continuous and
$\Lambda^{\alpha}$ is closed, we can find an $\varepsilon'>0$ such that no permuton $\gamma'$ within distance $\varepsilon'$ of $\Lambda^{\alpha}$ has $H(\gamma') > H(\gamma)+\delta$.
But again since $\rho(\pi,\cdot)$ is continuous, for small enough $\varepsilon$, every $\gamma' \in \Lambda^{\alpha,\varepsilon}$ is indeed within distance $\varepsilon'$ of $\Lambda^{\alpha}$.
Let $\Lambda'$ be the (closed) set of permutons $\gamma'$ satisfying $\rho(\pi_j,\gamma') \le \varepsilon$; then, using the first statement of Theorem~\ref{thm:tras}, we have
thus
\be
\lim_{n\to\infty}\frac1n \log \frac{|\Lambda'_n|}{n!} \le H(\gamma) + \delta
\ee
and since such a statement holds for arbitrary $\delta > 0$, the
result follows. \end{proof}

\section{Insertion measures}

A permuton $\gamma$ can be described by a family of \emph{insertion measures}. 
This description
will be useful for constructing concrete examples, in particular for the so-called star models, 
which are discussed in Sections
\ref{inversionsection} and 
\ref{starsection} below. 

The insertion measures are a family
of probability measures $\{\nu_x\}_{x\in[0,1]}$, with measure $\nu_x$ supported on $[0,x]$.
This family is a continuum version of the process of building a random permutation on $[n]$ by, for each $i$, inserting $i$ at a random location in the permutation
formed from $\{1,\dots,i-1\}$. Any permutation measure can be built this way. 
We describe here how any permuton can be built from a family of \emph{independent}
insertion measures, and conversely,
how every permuton defines a unique family of independent insertion measures.

We first describe how to reconstruct the insertion measures from the permuton $\gamma$.
Let $Y_x\in[0,1]$ be the random variable with law $\gamma|_{\{x\}\times[0,1]}$.
Let $Z_x\in[0,x]$ be the random variable (with law $\nu_x$) giving the location of the insertion of $x$ (at time $x$),
and let $F(x,\cdot)$ be its CDF.
Then 
\be
F(x,y) = \Pr(Z_x<y)=\Pr(Y_x<\tilde y)=G_x(x,\tilde y)
\ee
where $\tilde y$ is defined by $G(x,\tilde y)=y.$

More succinctly, we have 
\be \label{FfromG}F(x,G(x,\tilde y))=G_x(x,\tilde y).\ee

Conversely, given the insertion measures, equation (\ref{FfromG}) is a differential equation for $G$.
Concretely, after we insert $x_0$ at location $X(x_0) = Z_{x_0}$, the image flows under future insertions
according to the (deterministic) evolution
\be\label{VF} \frac{\dd}{\dd{}x}X(x) = F_x(X(x)),~~~X(x_0)=Z_{x_0}.\ee
If we let $\Psi_{[x,1]}$ denote the flow up until time $1$, then the permuton is the push-forward under $\Psi$ of $\nu_x$:
\be \gamma_t= (\Psi_{[x,1]})_*(\nu_x).\ee

A more geometric way to see this correspondence is as follows. 
Project the graph of $G$ in $\R^3$ onto the $xz$-plane; the image
of the curves $G([0,1]\times\{\tilde y\})$ are the flow lines of the vector field (\ref{VF}). The divergence of the flow lines
at $(x,y)$ is $f(x,y)$, the density associated with $F(x,y)$. 

The permuton entropy can be computed from the entropy of the insertion measures as follows.
\begin{lemma}
\be\label{fent}H(\gamma) = \int_0^1\int_0^x-f(x,y)\log(xf(x,y))\dd{}y\,\dd{}x.\ee
\end{lemma}

\begin{proof}Differentiating (\ref{FfromG}) with respect to $\tilde y$ gives 
\be f(x,G(x,\tilde y))G_y(x,\tilde y)=g(x,\tilde y).\ee
Thus the RHS of (\ref{fent}) becomes
\be \int_0^1\int_0^x-\frac{g(x,\tilde y)}{G_y(x,\tilde y)}\log\frac{x g(x,\tilde y)}{G_y(x,\tilde y)}\,\dd{}y\,\dd{}x.\ee
Substituting $y=G(x,\tilde y)$ with $\dd{}y = G_y(x,\tilde y)d\tilde y$ we have
\begin{eqnarray}
\int_0^1\int_0^1-g(x,\tilde y)\log\frac{x g(x,\tilde y)}{G_y(x,\tilde y)}\,d\tilde y\,\dd{}x
=H(\gamma) &-&\int_0^1\int_0^1g(x,\tilde y)\log x\,d\tilde y\,\dd{}x\cr
+\int_0^1\int_0^1g(x,\tilde y)\log G_y(x,\tilde y)\,d\tilde y\,\dd{}x.& &\end{eqnarray}
Integrating over $\tilde y$ the first integral on the RHS is 
\be \int_0^1-\log x\,\dd{}x = 1, \ee
 while the second integral is
\be \int_0^1\int_0^1 \frac{\partial}{\partial x}(G_y\log G_y-G_y)\,\dd{}x\,d\tilde y=\int_0^1(-1)d\tilde y= -1,\ee
since $G(1,y)=y$ and $G(0,y)=0$. So those two integrals cancel.
\end{proof}

\section{1\hh2 patterns}\label{inversionsection}

The number of occurrences $k(\pi)$ of the pattern $12$ in a permutation of $S_n$ has a simple generating function:
\be\label{12patternsgf}
\sum_{\pi\in S_n}x^{k(\pi)} = \prod_{j=1}^n(1+x+\dots+x^j) =
\sum_{i=0}^{\binom{n}{2}}C_i x^i.
\ee
One can see this by building up a permutation by insertions: when $i$ is inserted into the list of $\{1,\dots,i-1\}$,
the number of $12$ patterns created is exactly one less than the position of $i$ in that list. 

Theorem \ref{thm:constrained} suggests that to sample a permutation with a fixed density $\rho\in[0,1]$ of occurrences of
pattern $12$,
we should choose $x$ in the above expression so that the monomial $C_{[\rho n^2/2]}x^{[\rho n^2/2]}$ is the maximal one,
and then use the insertion probability measures which are (truncated)
geometric random variables with rate $x$.

Here $x$ is determined as a function of $\rho$ by Legendre duality (see below for an exact formula).
Let $r$ be defined by $e^{-r}=x$. In the limit of large $n$, the truncated geometric insertion 
densities converge to truncated exponential densities
\be f(x,y) = \frac{re^{-ry}}{1-e^{-rx}}\one_{[0,x]}(y).\ee

We can reconstruct the permuton from these insertion densities as follows. Note that the CDF of the insertion measure is
\be F(x,y) = \frac{1-e^{-ry}}{1-e^{-rx}}.\ee
We need to solve the ODE (\ref{FfromG}), which in this case (to simplify notation we changed $\tilde y$ to $y$) is
\be \frac{1-e^{-r G(x,y)}}{1-e^{-r x}} = \frac{\dd{}G(x,y)}{\dd{}x}.\ee
This can be rewritten as 
\be \frac{\dd{}x}{1-e^{-r x}} = \frac{\dd{}G}{1-e^{-r G(x,y)}}.\ee
Integrating both sides and solving for $G$ gives the CDF
\be\label{G12} G(x,y) = \frac1{r}\log\left(1+\frac{(e^{rx}-1)(e^{ry}-1)}{e^r-1}\right)\ee
which has density
\be\label{12g} g(x,y) = \frac{r(1-e^{-r})}{(e^{r(1-x-y)/2}-e^{r(x-y-1)/2}-e^{r(y-x-1)/2}+e^{r(x+y-1)/2})^2}.\ee
See Figure \ref{12permuton} for some examples for varying $\rho$.

The permuton entropy of this permuton is obtained from (\ref{fent}),
and as a function of $r$ it is, using the dilogarithm,
\be
H(r)=-\frac{2 \text{Li}_2\left(e^r\right)}{r}+\frac{\pi ^2}{3 r}-2 \log
\left(1-e^r\right)+\log \left(e^r-1\right)-\log (r)+2.\ee

The density $\rho$  of 1\hh2 patterns is the integral of the expectation of $f$:
\be
\rho(r) = \frac{r \left(r-2 \log \left(1-e^r\right)+2\right)-2 \text{Li}_2\left(e^r\right)}{
r^2}+\frac{\pi ^2}{3 r^2};\ee
see Figure \ref{inv} for $\rho$ as a function of $r$.

Figure \ref{entofinv} depicts the entropy as a function of $\rho$.

\begin{figure}[htbp]
\center{\includegraphics[width=3in]{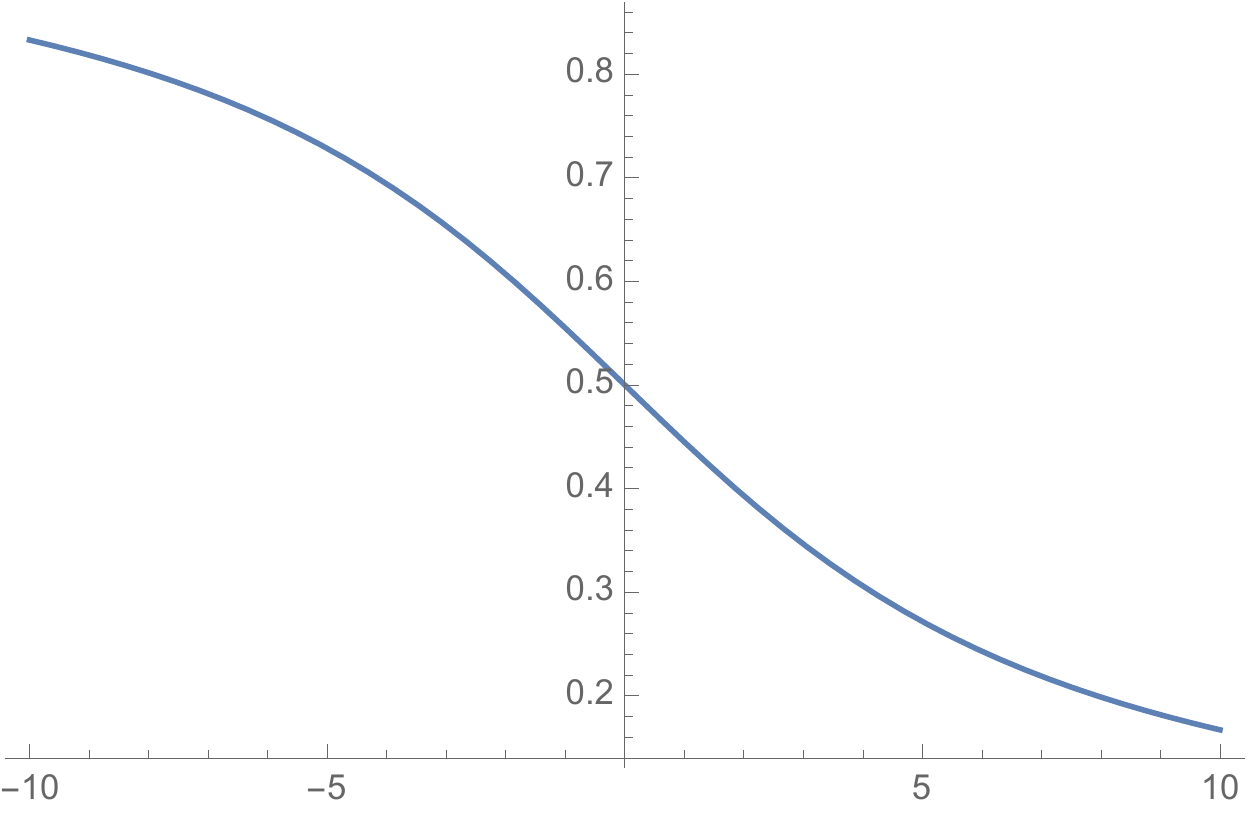}}
\caption{\label{inv}1\hh2 density as function of $r$.}
\end{figure}
\begin{figure}[htbp]
\center{\includegraphics[width=3in]{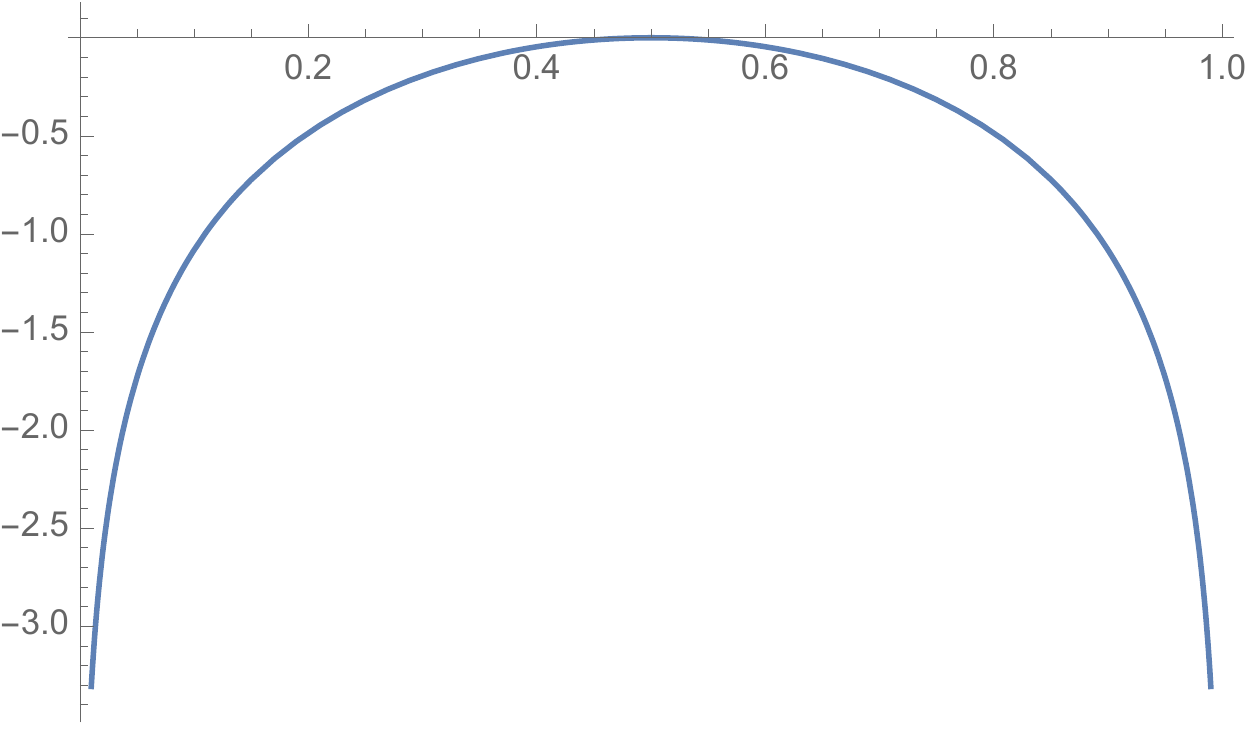}}
\caption{\label{entofinv}Entropy as function of 1\hh2 density.}
\end{figure}

\old{
Let us now compute the number of weak fixed points $FP(\gamma)$, that is, the limit as $n\to\infty$ of the expected number of fixed points in 
a size-$n$ permutation drawn from $\gamma_\rho$, where $\gamma_\rho$ is the above maximizing permuton with density $\rho$ of $12$ patterns. (Recall the definition of weak properties from the introduction.)
The probability that $(x,y)$ is a fixed point (conditioned on it being among the $n$ points chosen), 
is asymptotically zero unless $x=y$.
If $x=y$, it is a fixed point if the number of points with smaller $x$-coordinate equals the number of points with lower $y$ coordinate.
This is the central binomial coefficient $\binom{2m}{m}\approx \frac{1}{\sqrt{\pi m}}$ where 
$m$ is the number of points above and left of $(x,x)$. This leads to the formula (true for general absolutely continuous permutons)
\begin{proposition} For an absolutely continuous permuton,
\be\label{FPformula} FP(\gamma) = \int_0^1 \frac{g(x,x) \dd{}x}{\pi\sqrt{x-G(x,x)}}.\ee
\end{proposition}
Substituting (\ref{G12}) for $G$ yields the following plot of $FP(\gamma_\rho)$ against $\rho_{12}$.
\begin{figure}[htbp]
\center{\includegraphics[width=3in]{fixedpoints}}
\caption{\label{fixedpoints}The number of weak fixed points as a function of $\rho_{12}$. Note that 
it takes value $1$ at $\rho=1/2$ and tends to $\infty$ at $\rho=1$.}
\end{figure}

In the limit $\rho_{12}\to0$,  the number of weak fixed points tends to
$\frac1{\pi\sqrt{2}}\approx 0.225.$ 
}

\section{Star models}\label{starsection}
\newcommand{\st}{\ast\!}

In this section, we study the density of patterns of the form $\st\st\cdots\st k$,
which we refer to as star models.

Equation (\ref{12patternsgf}) gives the generating function for
occurrences of pattern $12$.
For a permutation $\pi$ let $k_1=k_1(\pi)$ be the number of $12$ patterns.
Let $k_2$ be the number of $\st\st3$ patterns, that is, patterns of the form $1\hh2\hh3$ or $2\hh1\hh3$.
A similar argument to that giving (\ref{12patternsgf}) 
shows that the joint generating function for $k_1$ and $k_2$ is
\be\label{k1k2} \sum_{k_1,k_2} C_{k_1,k_2}x^{k_1}y^{k_2} = \prod_{j=1}^n \left(\sum_{i=0}^j x^i y^{i(i-1)/2}\right).\ee
More generally, letting $k_3$ be the number of patterns $\st\st2$, that is, $1\hh3\hh2$ or $3\hh1\hh2$, and $k_4$ be the number of 
$\st\st1$ patterns, that is, $2\hh3\hh1$ or $3\hh2\hh1$. 
The joint generating function for these four types of patterns is
\be\label{biggf} \sum_{k_1,\ldots,k_4} C_{k_1,k_2,k_3,k_4}x^{k_1}y^{k_2}z^{k_3}w^{k_4} = 
\prod_{j=1}^n \left(\sum_{i=0}^j x^i y^{i(i-1)/2}z^{i(j-i)}w^{(j-i)(j-i-1)/2}\right).\ee
One can similarly write down the joint generating function for all patterns of type $\st\st\dots\st i$, with a string of some number $k$ of 
stars followed by some $i$ in $[k+1]$. (Note that with this notation,
$1\hh2$ patterns are $\ast2$ patterns.) These constitute a significant
generalization of the Mallows model {discussed in \cite{Starr}}. 

\subsection{The $\ast 2/\st\ast3$ model}

By way of illustration, let us consider the simplest case of $\ast2$ (that is, $1\hh2$) and $\st\st3$.

\begin{theorem}The feasible region for $(\rho_{\ast 2},\rho_{\ast\ast 3})$ is the region bounded below by the parameterized curve
\be (2t-t^2,3t^2-2t^3)_{t\in[0,1]}\ee 
and above by the parameterized curve 
\be (1-t^2,1-t^3)_{t\in[0,1]}.\ee
\end{theorem}

One can show that the permutons on the boundaries are 
unique and supported on line segments of slopes $\pm1$, and are as indicated in 
Figure~\ref{r1r2region}.
\begin{figure}[htbp]
\center\includegraphics[width=4in]{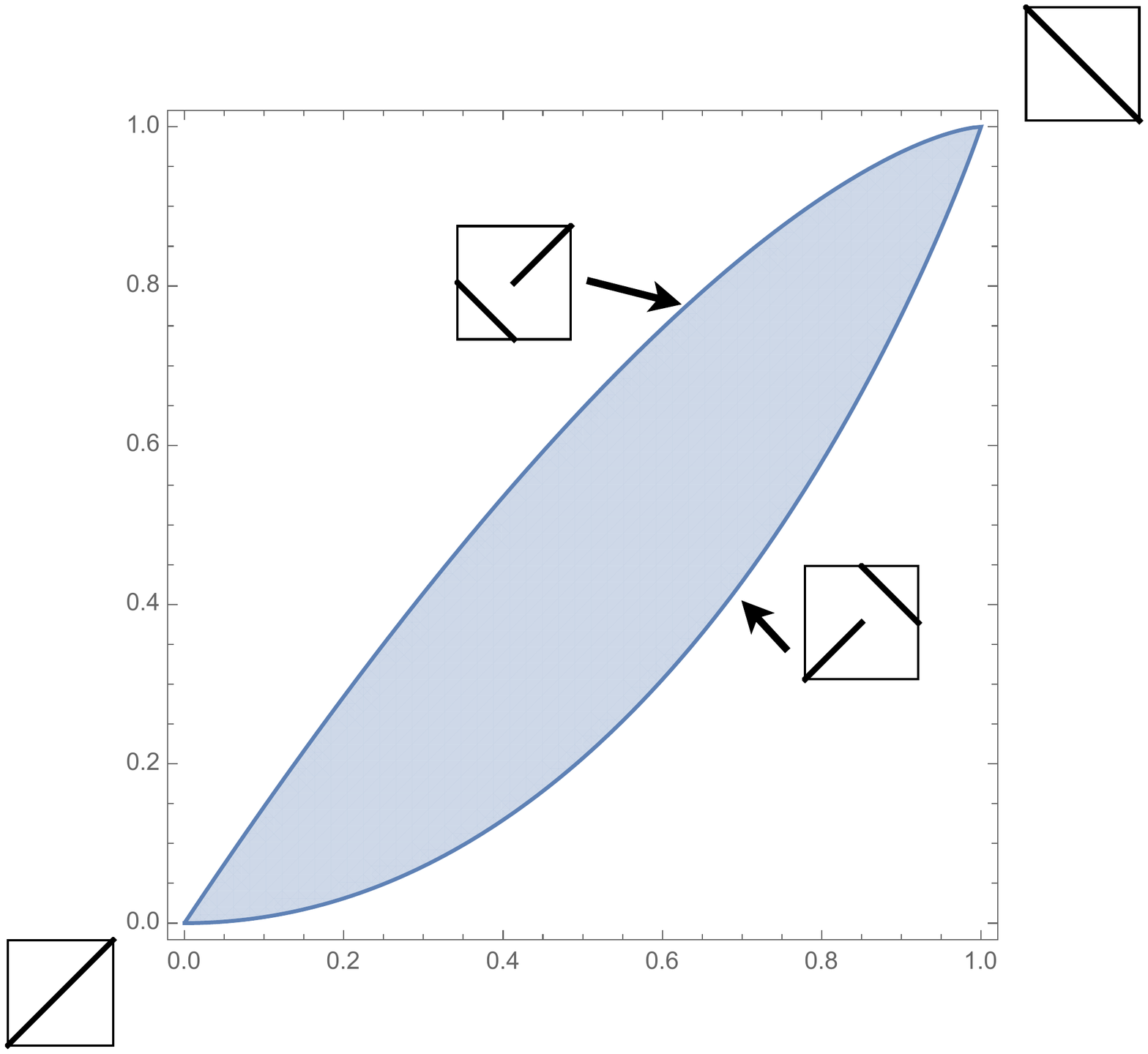}
\caption{\label{r1r2region}Feasible region for $(\rho_{\ast2},\rho_{\ast\ast3})$.}
\end{figure}

\begin{proof}While this can be proved directly from the generating function (\ref{k1k2}), we give a simpler proof using the 
insertion density procedure. During the insertion process let $I_{12}(x)$ be the fractional
number of $1\hh2$ patterns in the partial permutation constructed up to time $x$.
We want to stress the normalization factor here:
the number of $1\hh 2$ patterns in an $n$-permutation constructed up to time $x$ should be thought of as $I_{12}(x)n^2$,
in particular, $I_{12}(x)=\rho_{12}/2$. So, we get that $I_{12}(x) = \int_0^x Y_t\,\dd{}t$,
where $Y_t$ is the random variable giving the location of the insertion of $t$.
By the law of large numbers we can replace $Y_t$ here by its mean value,
that is, $I_{12}(x+\dd{}x) - I_{12}(x)$ is a sum (really, an integral) of the independent
insertions during time in $[x,x+\dd{}x]$, which have mean $Y_t$, so 
$$I_{12}'(t) = \E[Y_t].$$

Let $I_{\ast\ast3}(x)$ likewise be the fraction of $\st\st3$ patterns created by time $x$.
We have
\be I_{\ast\ast3}(x) = \int_0^x \E[Y_t]^2/2\,\dd{}t.\ee
Note that $I_{\ast\ast3}(x)=(\rho_{123}+\rho_{213})/6$.

\begin{figure}[htbp]
\center\includegraphics[width=2in]{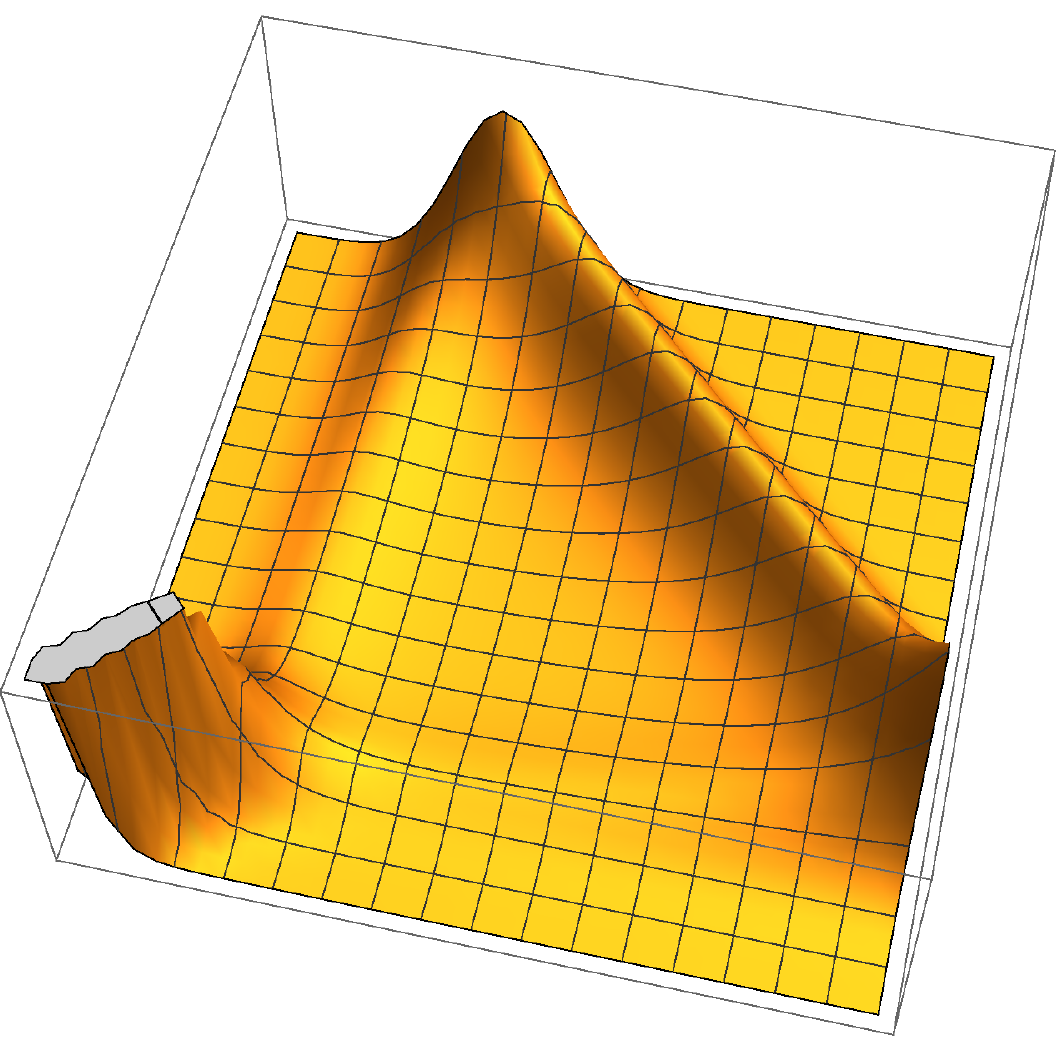}\includegraphics[width=2in]{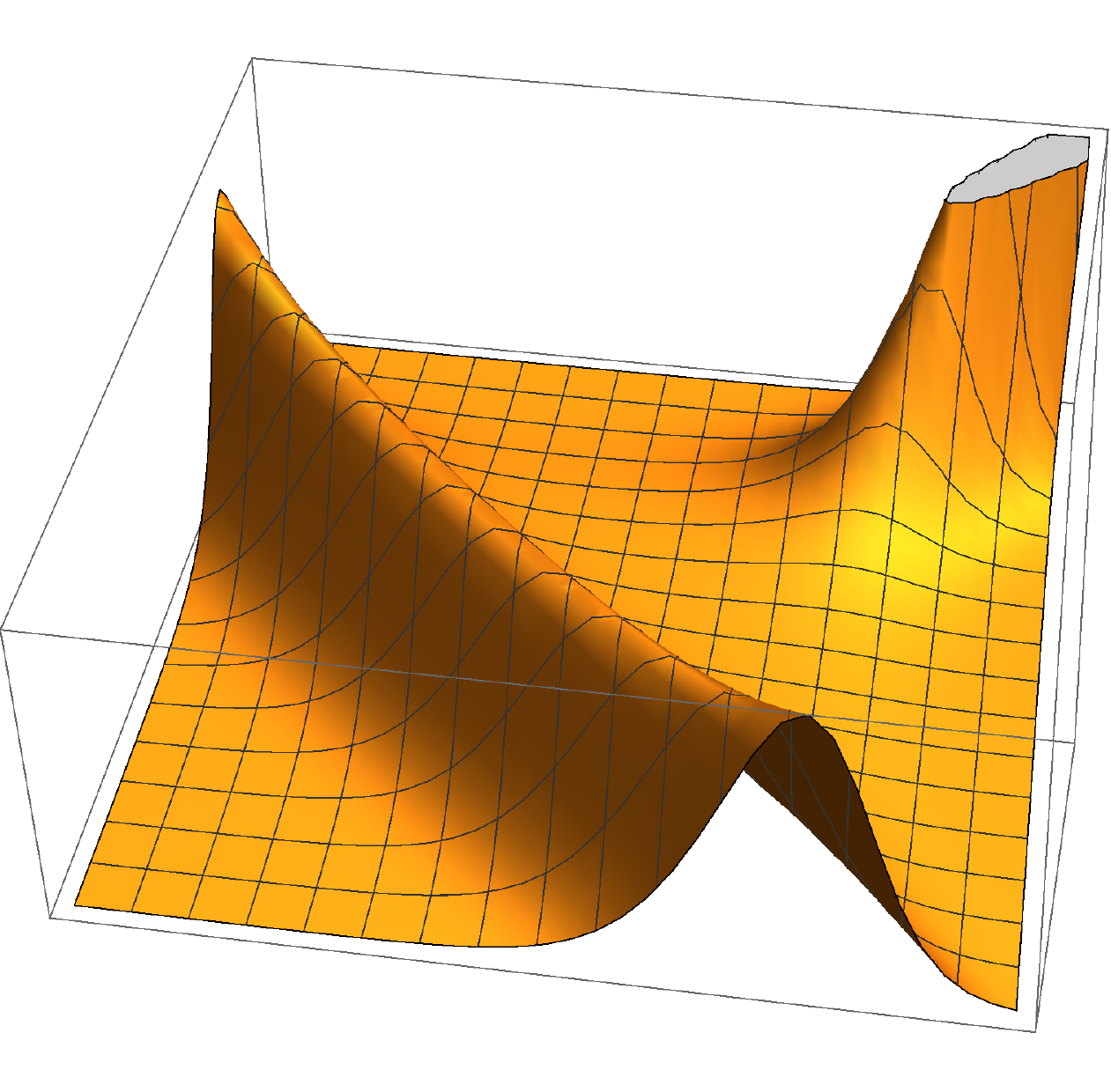}
\caption{\label{xypermutons}Permutons with $(\rho_{\ast2},\rho_{\ast\ast3})=(.5,.2),$ and $(.5,.53)$ respectively.}
\end{figure}

Let us fix $\rho_{12}=2\cdot I_{12}(1)$. To maximize $I_{\ast\ast3}(x)$, we need to maximize
\be\label{integralconstraint}\int_0^1 (I_{12}'(t))^2\,\dd{}t \text{~~~subject to~~~}\int_0^1 I_{12}'(t)\,\dd{}t=
\frac{\rho_{12}}2.\ee
This is achieved by making $I_{12}'(t)$ either zero or maximal. Since $I_{12}'(t)\le t,$ we can achieve
this by inserting points at the beginning
for as long as possible and then inserting points at the end, that is, $Y_t=0$ up to $t=a$ and then $Y_t=t$ for $t\in[a,1]$. 
The resulting permuton is then as shown in Figure \ref{r1r2region}:
on the square $[0,a]^2$ it is a descending diagonal and on the square $[a,1]^2$ it is an ascending diagonal.

Likewise to minimize the above integral (\ref{integralconstraint}) we need to make the derivatives $I_{12}'(t)$ as equal as possible.
Since $I_{12}'(t)\le t$, this involves setting $I_{12}'(t)=t$ up to $t=a$ and then having it constant after that.
The resulting permuton is then as shown in Figure \ref{r1r2region}:
on the square $[0,a]^2$ it is an ascending diagonal and on the square $[a,1]^2$ it is a descending diagonal.

A short calculation now yields the algebraic form of the boundary curves.
\end{proof}

Using the insertion density procedure outlined earlier, we see that the permuton as a function of $x,y$ has an explicit analytic density
(which cannot, however, be written in terms of elementary functions). 
The permutons for some values of $(\rho_{\ast2},\rho_{\ast\ast3})$
are shown in Figure \ref{xypermutons}.

The entropy $s(\rho_{\ast2},\rho_{\ast\ast3})$ is plotted in Figure \ref{entropy1x1xx}. 
It is strictly concave (see Theorem \ref{cov} below)
and achieves
its maximal value, zero, precisely
at the point $1/2,1/3$, the uniform measure.
\begin{figure}[htbp]
\center\includegraphics[width=5.4in]{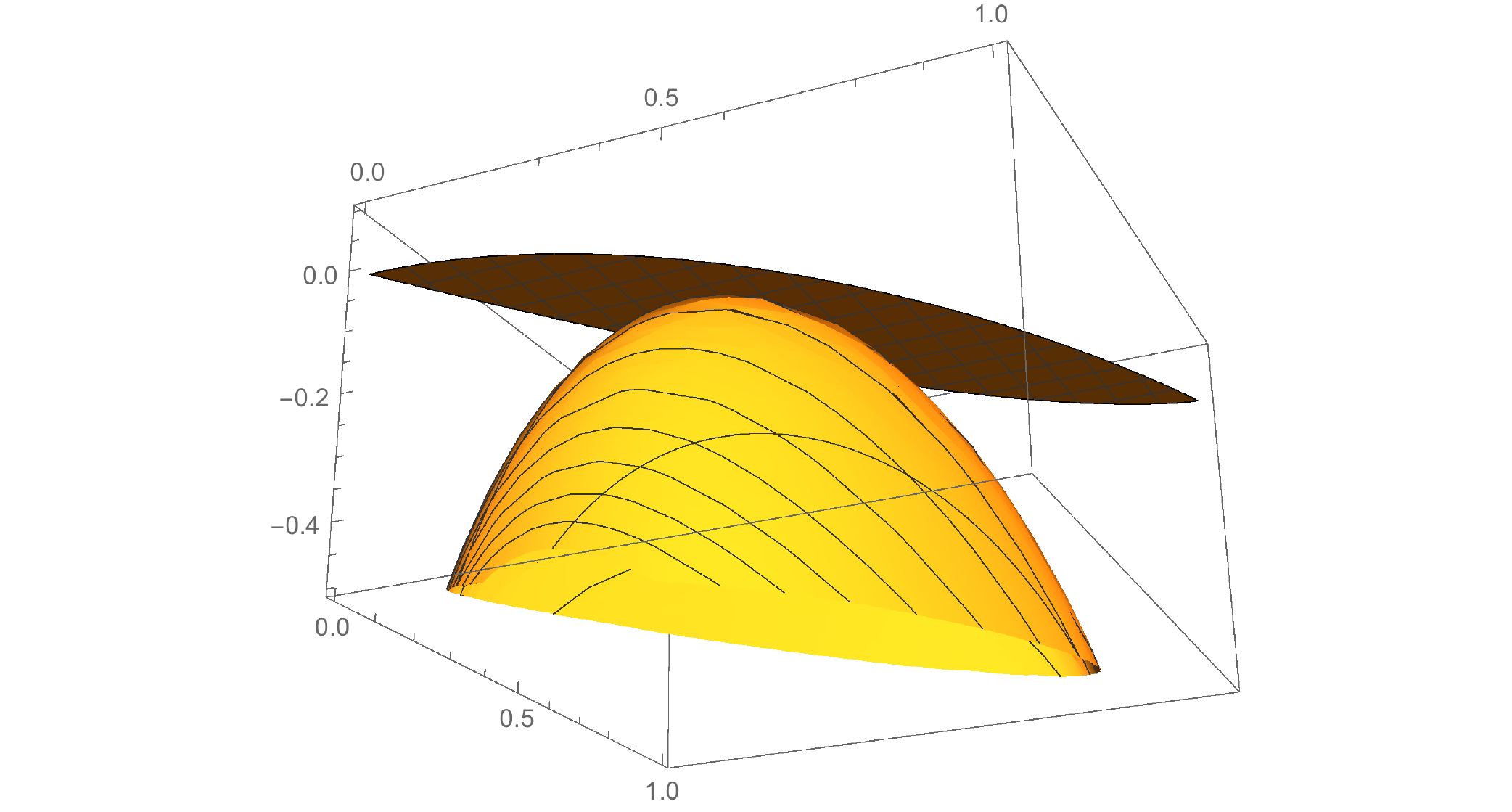}
\caption{\label{entropy1x1xx}The entropy function on the parameter space for $\rho_{12},\rho_{\st\st3}$.}
\end{figure}

\subsection{Concavity and analyticity of entropy for star models}

\begin{theorem}\label{cov}
For a star model with a finite number of densities $\rho_1,\dots,\rho_k$ of patterns $\tau_1\dots,\tau_k$ respectively,
the feasible region is convex and the entropy $H(\rho_1,\dots,\rho_k)$
is strictly concave and analytic on the feasible region. For each $\rho_1,\dots,\rho_k$ in the interior of the feasible
region there is a unique entropy-maximizing permuton with those densities, and this permuton has analytic 
probability density.
\end{theorem}

One can construct examples where the feasible region is not strictly convex: e.g. in the case of densities
$\st\st1$ and $\st\st3$.

\begin{proof}
Let $k_i$ be the length of the pattern $\tau_i$.

The generating function for permutations of $[n]$ counting patterns $\tau_i$ is
\be Z_n(x_1,\dots,x_k) = \sum_{\pi\in S_n} x_1^{n_1}\dots x_k^{n_k}\ee
where $n_i=n_i(\pi)$ is the number of occurrences of pattern $\tau_i$ in $\pi$. 
{The number of permutations with density $\rho_i$ of pattern $\tau_i$ 
is the sum of the coefficients of the terms $x_1^{n_1}\dots x_k^{n_k}$ with
$n_i\approx \frac{n^{k_i}}{k_i!}\rho_i$. The entropy $H(\rho_1,\dots,\rho_k)$ is the log of this sum,
minus $\log n!$ (and normalized by dividing by $n$).}

As discussed above, $Z_n$ can be written as a product generalizing (\ref{biggf}). 
Write $x_i = e^{a_i}$.
Then the product expression for $Z_n$ is
\be Z_n = \prod_{j=1}^n\sum_{i=0}^j e^{p(i,j)},\ee
where $p(i,j)$ is a polynomial in $i$ and $j$ with coefficients that are linear in the $a_i$. 
For large $n$ it is convenient to normalize the $a_i$ by an appropriate power of $n$ (and a combinatorial factor):
write 
\be x_i = e^{a_i} = \exp{\left(\alpha_i/n^{k_i-1}\right)}.\ee

Writing $i/n=t$ and $j/n=x$, the expression for $\log Z_n$ is then a Riemann sum, once normalized:
In the limit $n\to\infty$ the ``normalized free energy'' $F$ is
\be F:=\lim_{n\to\infty} \frac1n(\log Z_n - \log n!) = \int_0^1\left[\log\int_0^xe^{\tilde p(t,x)}\,\dd{}t\right]\,\dd{}x\ee
where $\tilde p(t,x)= p(nt,nx)+o(1)$ is a polynomial in $t$ and $x$, independent of $n$, with coefficients 
which are linear functions of the $\alpha_i$. Explicitly we have 
\be \tilde p(t,x) = \sum_{i=1}^k  \alpha_i\frac{t^{r_i}(x-t)^{s_i}}{r_i!s_i!}\ee
where $r_i+s_i=k_i-1$ and, if $\tau_i = \st\dots\st\ell_i$ then $s_i=k_i-\ell_i.$

We now show that $F$ is concave as a function of the $\alpha_i$, by computing its Hessian matrix. We have

\be \frac{\partial F}{\partial\alpha_i} = 
\int_0^1\frac{\int_0^x \frac{t^{r_i}(x-t)^{s_i}}{r_i!s_i!}e^{\tilde p(t,x)}\dd{}t}{\int_0^x e^{\tilde p(t,x)}\dd{}t} \dd{}x 
=\int_0^1 \left\langle \frac{T^{r_i}(x-T)^{s_i}}{r_i!s_i!}\right\rangle \dd{}x\ee
where $T\in[0,x]$ is the random variable with (unnormalized) density
$e^{\tilde p(t,x)}$, and
$\left\langle\cdot\right\rangle$ is the expectation with respect to this probability 
measure. 

Differentiating a second time we have
\begin{eqnarray} \frac{\partial^2 F}{\partial\alpha_j\partial \alpha_i} & = &
\int_0^1 \left\langle \frac{T^{r_i+r_j}(x-T)^{s_i+s_j}}{r_i!r_j!s_i!s_j!}\right\rangle-
\left\langle \frac{T^{r_i}(x-T)^{s_i}}{r_i!s_i!} \right\rangle\left\langle \frac{T^{r_j}(x-T)^{s_j}}{r_j!s_j!}\right\rangle \dd{}x\cr
&=&\int_0^1 \text{Cov}\left[\frac{T^{r_i}(x-T)^{s_i}}{r_i!s_i!},\frac{T^{r_j}(x-T)^{s_j}}{r_j!s_j!}\right]\,\dd{}x\end{eqnarray}
where $\text{Cov}$ is the covariance.

The covariance matrix of a set of random variables with no linear dependencies
is positive definite. 
Thus we see that the Hessian matrix is an integral
of positive definite matrices and so is itself positive definite. 
This completes the proof of strict concavity of the free energy $F$. 

Since $Z_n$ is the (unnormalized) probability generating function, the vector of densities
as a function of the $\{\alpha_i\}$ is obtained for each $n$ by the gradient of the logarithm
\be\label{Leg} (\rho_1,\dots,\rho_k) = \frac1n\nabla \log Z_n(\alpha_1,\dots,\alpha_k).\ee
In the limit we can replace $\frac1n\nabla\log Z_n$ by $\nabla F$; by strict concavity of $F$ its gradient is injective,
and surjective onto the interior of the feasible region. 
In particular there is a unique choice of $\alpha_i$'s for every choice of densities in the interior of the feasible region.
{Note that the $\alpha_i$'s determine the insertion measures (these are the measures with unnormalized density $e^{\tilde p(t,x)}$), and thus the permuton itself,  
proving uniqueness of the entropy maximizer}. Analyticity of the probability density is a consequence of analyticity
of the associated differential equation \eqref{FfromG}. 

{
By strict concavity of the free energy, we can relate the free energy to the entropy by the following
standard argument.
Referring back to the first paragraph of the proof, we have proven that, when $x_i=e^{a_i}$, 
the generating function $Z_n$ concentrates
its mass on the terms $x_1^{n_1}\dots x_k^{n_k}$ for which $n_i\approx \frac{n^{k_i}}{k_i!}\rho_i$ 
(where $a_i$ and $\rho_i$
are related by (\ref{Leg})), in the sense
that a fraction $1-o(1)$ of the total mass of $Z_n$ is on these terms. The entropy is the log of the sum of the coefficients
in front of these relevant terms. The entropy can thus be obtained from the free energy 
$\frac1n\log Z_n/n!$
by subtracting off $\frac1n\log(x_1^{n_1}\dots x_k^{n_k}).$
This shows that the entropy function $H$ is the Legendre dual of $F$,
that is, 
\be H(\rho_1,\dots,\rho_k) = \max_{\{\alpha_i\}}\{F(\alpha_1,\dots,\alpha_k) -\sum\alpha_i\rho_i \}.\ee
Analyticity of $F$ implies that $H$ is both analytic and strictly concave. 

The ``upper level sets'' $\{\vec\rho~:~H(\vec\rho)\ge -M\}$ of $H$ are convex by concavity of $H$. 
Their union is the interior of the feasible region, which, being an increasing
union of convex sets, is convex.}
\end{proof}

\section{PDEs for permutons}\label{PDEsection}

For permutations with constraints on patterns of length $3$ (or less) one can write explicit PDEs for the maximizers.
It is possible that these may be used to show either analyticity or uniqueness, or both (although we have accomplished neither goal).

Let us first redo the case of $1\hh2$-patterns, which we already worked out by another method in Section \ref{inversionsection}.

\subsection{Patterns $1\hh2$}

The density of patterns $1\hh2$ is given in (\ref{2}). 
Consider the problem of maximizing $H(\gamma)$ subject to the constraint $I_{12}(\gamma)=\rho$.
This involves finding a solution to the Euler-Lagrange equation 
\be \label{EL1}\dd{}H+\alpha\, \dd{}I_{12}=0\ee
for some constant $\alpha$,
for all variations $g\mapsto g+\eps h$ fixing the marginals.

Given points $(a_1,b_1),(a_2,b_2)\in[0,1]^2$ we can consider the change in $H$ and $I_{12}$ when we
remove an infinitesimal mass $\delta$ from $(a_1,b_1)$ and $(a_2,b_2)$ and add it to locations $(a_1,b_2)$ and $(a_2,b_1)$. 
(Note that two measures with the same marginals are connected by convolutions of such operations.)
The change in $H$ to first order under such an operation is $\delta$ times (letting $S_0(p):=-p\log p$)
\begin{multline}-S_0'(g(a_1,b_1))-S_0'(g(a_2,b_2))+S_0'(g(a_1,b_2))+S_0'(g(a_2,b_1))\\
=\log\frac{g(a_1,b_1)g(a_2,b_2)}{g(a_1,b_2)g(a_2,b_1)}.\end{multline}

The change in $I_{12}$ to first order is $\delta$ times
\begin{multline}
\sum_{i,j=1}^2 (-1)^{i+j}\left(\int_{a_i<x_2}\int_{b_j<y_2} g(x_2,y_2)\dd{}x_2\,\dd{}y_2+\int_{x_1<a_i}\int_{y_1<b_j} g(x_1,y_1)\dd{}x_1\,\dd{}y_1\right)\\
=\sum_{i,j=1}^2 (-1)^{i+j} \left(G(a_i,b_j) + (1-a_i-b_j+G(a_i,b_j))\right).
\end{multline}

Differentiating (\ref{EL1}) with respect to $a=a_1$ and $b=b_1$,
we find 
\be \frac{\partial}{\partial a}\frac{\partial}{\partial b}\log g(a,b)+2\alpha g(a,b)=0.\ee
One can check that the formula (\ref{12g}) satisfies this PDE.

\subsection{Patterns $1\hh2\hh3$}
The density of patterns $1\hh2\hh3$  is
\be I_{123}(\gamma) = 6\int_{x_1<x_2<x_3,~y_1<y_2<y_3}g(x_1,y_1)g(x_2,y_2)g(x_3,y_3)\dd{}x_1\,\dd{}x_2\cdots\,\dd{}y_3.\ee
Under a similar perturbation as above the change in $I_{123}$ to first order is $\delta$ times

\begin{multline}\dd{}I_{123}=
  6\sum_{i,j=1}^2(-1)^{i+j}\left(\int_{a_i<x_2<x_3,~b_j<y_2<y_3}g(x_2,y_2)g(x_3,y_3)\dd{}x_2\,\dd{}x_3\,\dd{}y_2\,\dd{}y_3\right.\\
+\int_{x_1<a_i<x_3,~y_1<b_j<y_3}g(x_1,y_1)g(x_3,y_3)\dd{}x_1\,\dd{}x_3\,\dd{}y_1\,\dd{}y_3\\
  +\left.\int_{x_1<x_2<a_i,~y_1<y_2<b_j}g(x_1,y_1)g(x_2,y_2)\dd{}x_1\,\dd{}x_2\,\dd{}y_1\,\dd{}y_2\right).
  \end{multline}

The middle integral here is a product 
\begin{multline}
\int_{x_1<a_i,~y_1<b_j}g(x_1,y_1)\dd{}x_1\,\dd{}y_1\int_{a_i<x_3,~b_j<y_3}g(x_3,y_3)\dd{}x_3\,\dd{}y_3 \\
= G(a_i,b_j)(1-a_i-b_j+G(a_i,b_j)).
\end{multline}

Differentiating each of these three integrals with respect to both $a=a_1$ and $b=b_1$ (then only the $i=j=1$ term survives) 
gives, for the first integral
\be g(a,b)\int_{a<x_3,~b<y_3}g(x_3,y_3)\dd{}x_3\,\dd{}y_3=g(a,b)(1-a-b+G(a,b)),\ee
for the second integral
\begin{multline}
g(a,b)(1-a-b+2G(a,b))+G_x(a,b)(-1+G_y(a,b))\\
+G_y(a,b)(-1+G_x(a,b)),
\end{multline}
and the third integral
\be g(a,b)\int_{x_1<a,~b<y_1}g(x_1,y_1)\dd{}x_1\,\dd{}y_1=g(a,b)G(a,b).\ee
 
Summing, we get (changing $a,b$ to $x,y$)
\be (\dd{}I_{123})_{xy} = 12G_{xy}(1-x-y+2G)+12G_xG_y-6G_x-6G_y.\ee

Thus the Euler-Lagrange equation is
\be (\log G_{xy})_{xy}+6\alpha\big(2G_{xy}(1-x-y+2G)+2G_xG_y-G_x-G_y\big)=0.\ee

This simplifies somewhat if we define $K(x,y) = 2G(x,y)-x-y+1.$ Then 
\be (\log K_{xy})_{xy}+3\alpha\left(2K_{xy}K+K_xK_y-1\right)=0.\ee

In a similar manner we can find a PDE for the permuton with fixed densities of other patterns of length $3$.
In fact one can proceed similarly for longer patterns, getting systems of PDEs, but the complexity grows with the length.

\section{The $1\hh2/1\hh2\hh3$ model}
When we fix the density of patterns $1\hh2$ and $1\hh2\hh3$, the feasible region has a complicated structure, see Figure \ref{scallopedtri}.
\begin{figure}[htbp]
\center{\includegraphics[width=5.4in]{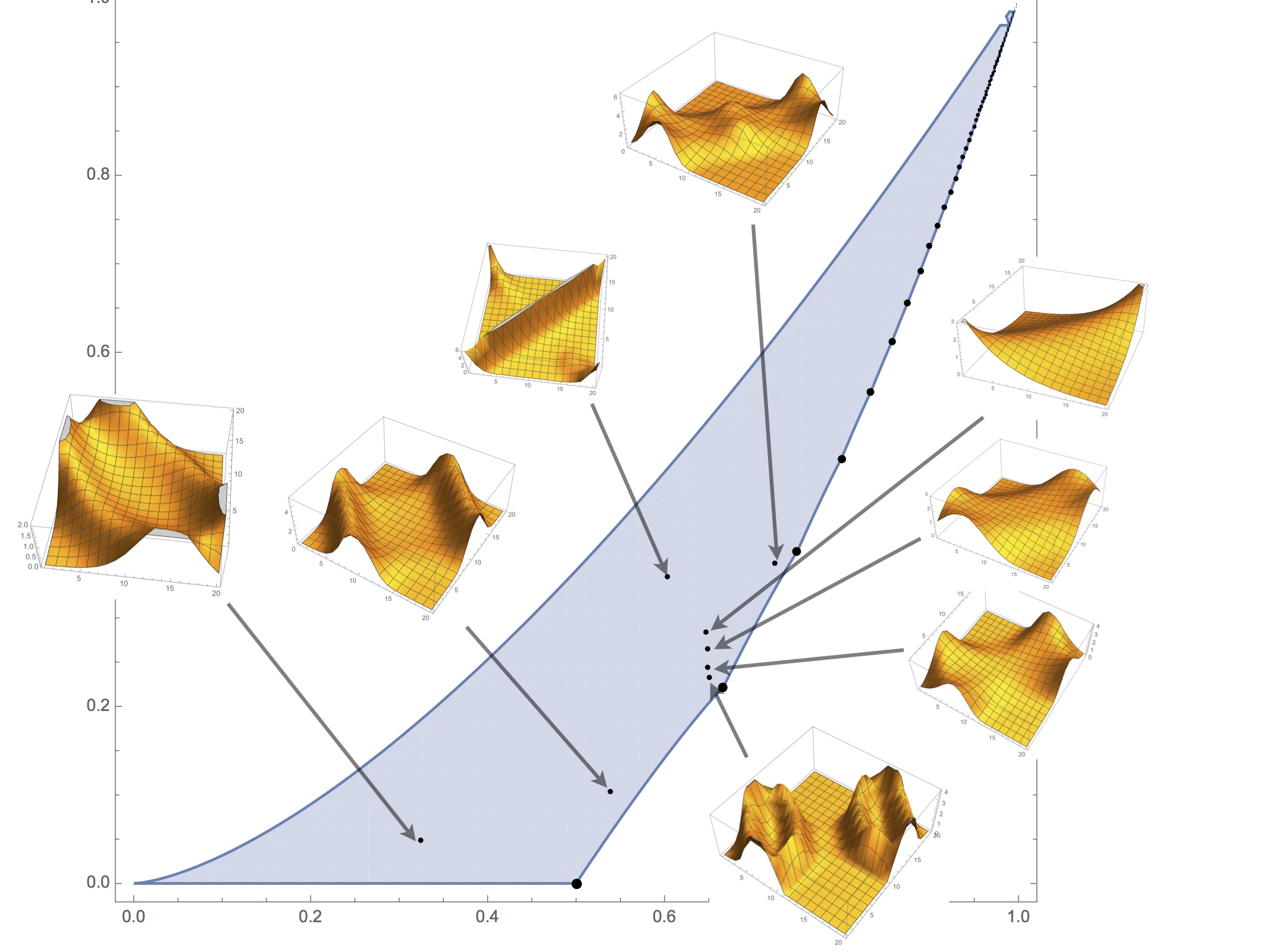}}
\caption{\label{scallopedtri}The feasible region for $\rho_{12}$ versus $\rho_{123}$, with corresponding permutons 
(computed numerically) at selected points.}
\end{figure}

\begin{theorem}
The feasible region for $\rho_{12}$ versus $\rho_{123}$ is the same as the feasible region of edges and triangles
in the graphon model.
\end{theorem}

\begin{proof}
Let $\cal R$ denote the feasible region for pairs $\big(\rho_{12}(\gamma),\rho_{123}(\gamma)\big)$ consisting of
the 1\hh2 density and 1\hh2\hh3 density of a permuton (equivalently, for the closure of the set of such pairs
for finite permutations).

Each permutation $\pi \in S_n$ determines a (two-dimensional) poset $P_\pi$ on $\{1,\dots,n\}$
given by $i \prec j$ in $P_\pi$ iff $i < j$ and $\pi_i < \pi_j$.  The {\em comparability graph} $G(P)$
of a poset $P$ links two points if they are comparable in $P$, that is, $x \sim y$ if $x \prec y$ or $y \prec x$.
Then $i \sim j$ in $G(P_\pi)$ precisely when $\{i,j\}$ constitutes an incidence of the pattern 1\hh2,
and $i \sim j \sim k \sim i$ when $\{i,j,k\}$ constitutes an incidence of the pattern 1\hh2\hh3.
Thus the 1\hh2 density of $\pi$ is equal to the edge density of $G(P_\pi)$, and the 1\hh2\hh3 density of $\pi$ is
the triangle density of $G(P_\pi)$---that is, the probability that three random vertices induce the
complete graph $K_3$.  This correspondence extends perfectly to limit objects, equating 1\hh2 and 1\hh2\hh3 densities
of permutons to edge densities and triangle densities of graphons.

The feasible region for edge and triangle densities of graphs (now, for graphons) has been studied for many years and
was finally determined by Razborov \cite{R}; we call it the ``scalloped triangle'' $\cal T$. It follows from the above discussion
that the feasibility region $\cal R$ we seek for permutons is a subset of $\cal T$, and it remains only to prove that $\cal R$ is
all of $\cal T$.  In fact we can realize $\cal T$ using only a rather simple two-parameter family of permutons.

Let reals $a$, $b$ satisfy $0 < a \le 1$ and $0 < b \le a/2$, and set $k := \lfloor a/b \rfloor$.  Let us denote by $\gamma_{a,b}$
the permuton consisting of the following diagonal line segments, all of equal density: 
\begin{enumerate}
\item The segment $y = 1-x$, for $0 \le x \le 1{-}a$;
\item The $k$ segments $y = (2j{-}1)b-1+a-x$ for $1-a+(j{-}1)b < x \le 1-a+jb$, for each $j = 1,2,\dots,k$; 
\item The remaining, rightmost segment $y = 1+kb-x$, for $1-a+kb < x \le 1$.

\end{enumerate}
(See Fig.~\ref{fig:abperm} below.)

\begin{figure}[htbp]
\center{\includegraphics[width=3.5in]{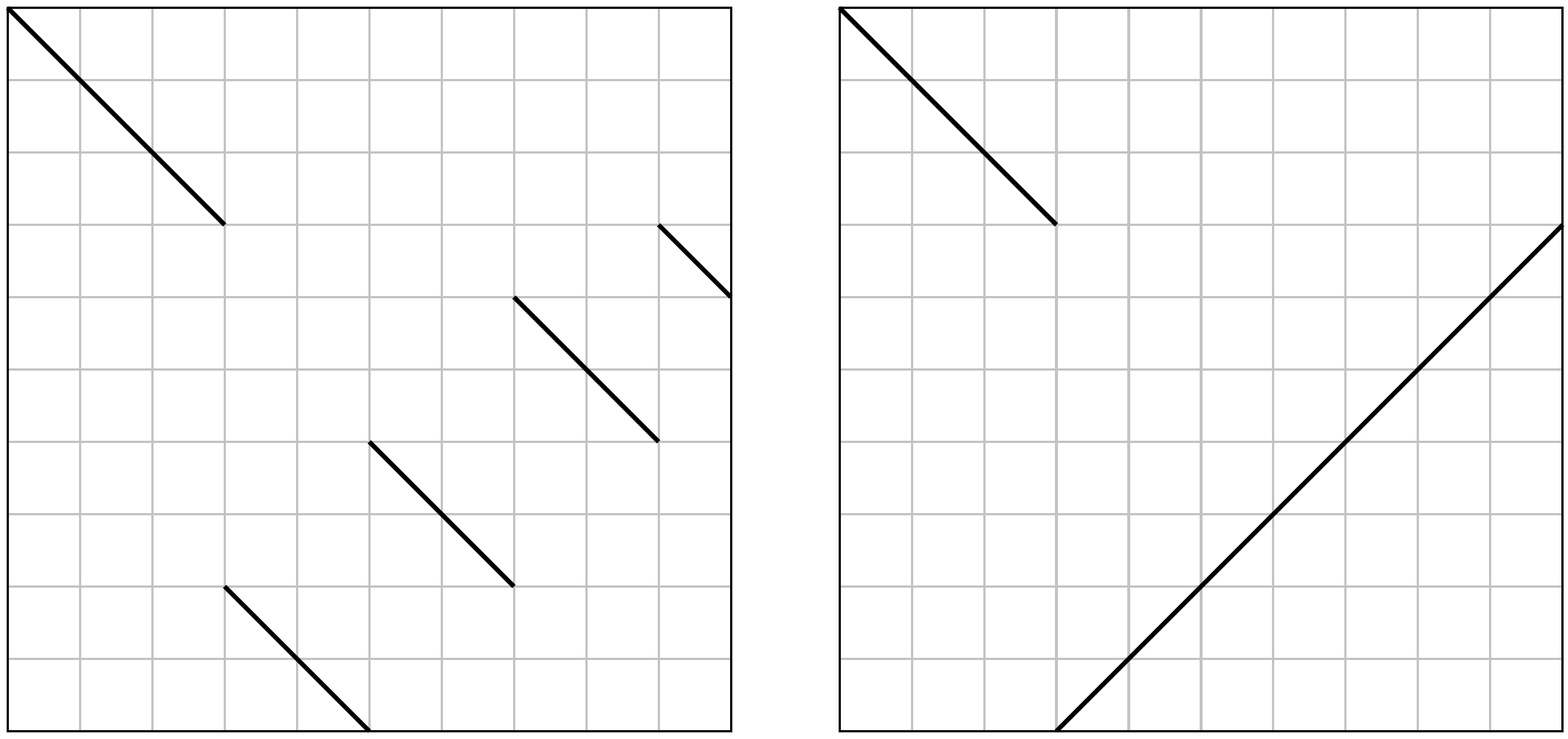}}\vskip-3cm
\caption{\label{fig:abperm}Support of the permutons $\gamma_{.7,.2}$ and $\gamma_{.7,0}$.}
\end{figure}

We interpret $\gamma_{a,0}$ as the permuton containing the segment $y = 1-x$, for $0 \le x \le 1{-}a$, and
the positive-slope diagonal from $(1{-}a,0)$ to $(1,1{-}a)$; finally, $\gamma_{0,0}$ is just the reverse
diagonal from $(0,1)$ to $(1,0)$.  These interpretations are consistent in the sense that $\rho_{12}(\gamma_{a,b})$ and
$\rho_{123}(\gamma_{a,b})$ are continuous functions of $a$ and $b$ on the triangle $0 \le a \le 1$, $0 \le b \le a/2$.
(In fact, $\gamma_{a,b}$ is itself continuous in the topology of $\Gamma$, so all pattern densities are continuous.)

It remains only to check that the comparability graphons corresponding to these permutons match extremal graphs
in \cite{R} as follows:
\begin{itemize}
\item $\gamma_{a,0}$ maps to the upper left boundary of $\cal T$, with $\gamma_{0,0}$ going to the lower left corner
while $\gamma_{1,0}$ goes to the top;
\item $\gamma_{a,a/2}$ goes to the bottom line, with $\gamma_{1,1/2}$ going to the lower right corner;
\item For $1/(k{+}2) \le b \le 1/(k{+}1)$, $\gamma_{1,b}$ goes to the $k$th lowest scallop, with
$\gamma_{1,1/(k+1)}$ going to the bottom cusp of the scallop and $\gamma_{1,1/(k+2)}$ to the top.
\end{itemize}

It follows that $(a,b) \mapsto \big(\rho_{12}(\gamma_{a,b}),\rho_{123}(\gamma_{a,b})\big)$ maps the triangle
$0 \le a \le 1$, $0 \le b \le a/2$ onto all of $\cal T$, proving the theorem.
\end{proof}

It may be prudent to remark at this point that while the feasible region for 1\hh2 versus 1\hh2\hh3 density of permutons
is the same as that for edge and triangle density of graphs, the {\em topography} of the corresponding entropy functions
within this region is entirely different.  In the graph case the entropy landscape is studied in \cite{RS1,RS2,RRS}; one of
its features is a ridge along the ``Erd\H{o}s-R\'{e}nyi'' curve (where triangle density is the 3/2 power of edge
density).  There is a sharp drop-off below this line, which represents the very high entropy graphs constructed
by choosing edges independently with constant probability.   The graphons that maximize entropy at each
point of the feasible region all appear to be very combinatorial in nature: each has a partition of its
vertices into finitely many classes, with constant edge density between any two classes and within any class,
and is thus described by a finite list of real parameters.

The permuton topography features a different high curve, representing the permutons (discussed above) that
maximize entropy for a fixed 1\hh2 density.  Moreover, the permutons that maximize entropy at interior points
of the region appear, as in other regions discussed above, always to be analytic.

We do not know explicitly the maximizing permutons (although they satisfy an explicit PDE, see Section \ref{PDEsection})
or the entropy function. 

\section{$1\hh2\hh3/3\hh2\hh1$ case}

The feasible region for
fixed densities $\rho_{123}$ versus $\rho_{321}$ is the same as
\begin{figure}[htbp]
\center{\includegraphics[width=5.4in]{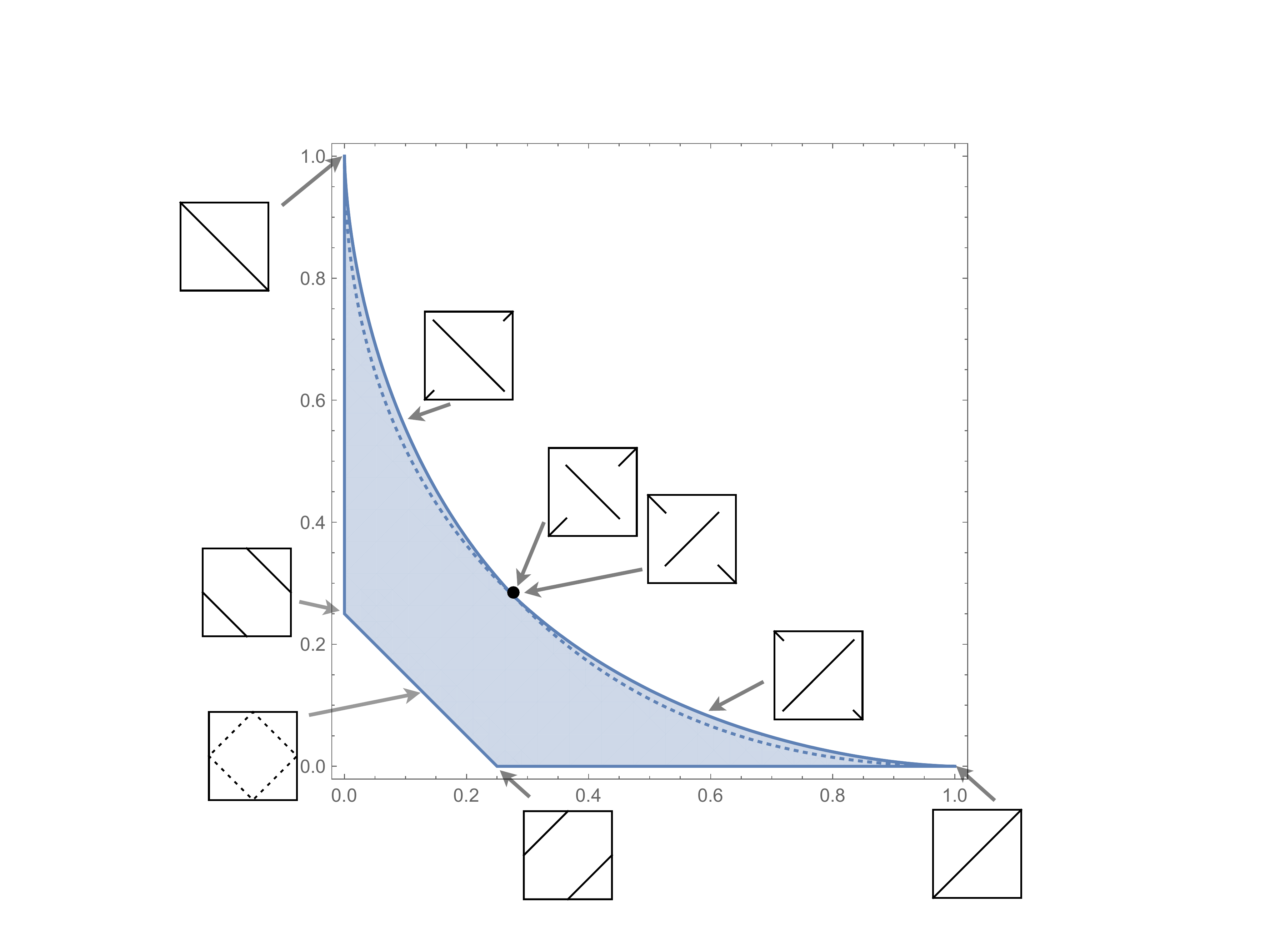}}
\caption{\label{123321region}The feasible region for $\rho_{123},\rho_{321}$. It is bounded above by the parameterized curves
$(1 - 3 t^2 + 2 t^3,t^3)$ and $(t^3, 1 - 3 t^2 + 2 t^3)$ which intersect at $(x,y)=(.278...,.278...)$. The lower boundaries
consist of the axes and the line $x+y=1/4$.}
\end{figure}
the feasible region $\cal B$ for triangle density $x = d(K_3,G)$
versus anti-triangle density $y = d(\overline{K_3},G)$
of graphons  \cite{HLNPS2}.  Let $C$ be the line
segment $x+y=\frac14$ for $0 \le x \le \frac14$,
$D$ the $x$-axis from $ x = \frac14$ to  $x = 1$, and $E$ the $y$-axis
from $y=\frac14$ to $y=1$.
Let $F_1$ be the curve given parametrically by $(x,y) = (t^3, (1-t)^3
+ 3t(1-t)^2)$, for $0 \le t \le 1$,
and $F_2$ its symmetric twin $(x,y) = ((1-t)^3 + 3t(1-t)^2, t^3)$.
Then $\cal B$ is the union of the
area bounded by $C$, $D$, $E$ and $F_1$ and the area bounded by $C$,
$D$, $E$ and $F_2$.

The curves $F_1$ and $F_2$ cross at a concave ``dimple'' $(r,r)$ where
$r = s^3 = (1-s)^3 + 3s(1-s)^2)$,
with $s \sim .653$ and $r \sim .278$; see Fig.~\ref{123321region}.

To see that $\cal B$ is also the feasible region for 1\hh2\hh3 versus 3\hh2\hh1
density of permutons, an argument much
like the one above for 1\hh2 versus 1\hh2\hh3
can be (and was, by
\cite{EN}) given.   Permutons realizing
various boundary points are illustrated in Fig.~\ref{123321region};
they correspond to the extremal
graphons described in \cite{HLNPS2}.  The rest are filled in by
parameterization and a topological argument.

Of note for both graphons and permutons is the double solution at the
dimple.  These solutions are significantly
different, as evidenced by the fact that their edge-densities (1\hh2
densities, for the permutons) differ.
This multiplicity of solutions, if there are no permutons bridging the gap,
suggests a phase transition in the entropy-optimal permuton in the interior of
$\cal B$ in a neighborhood of the dimple.  In fact, we can use a
stability theorem from \cite{HLNPS1} to show that
the phenomenon is real.

Before stating the next theorem, we need a definition:
two $n$-vertex graphs are \emph{$\varepsilon$-close}
if one can be made isomorphic to the other by adding or deleting at most $\varepsilon \cdot {n \choose 2}$ edges.

\begin{theorem}[special case of Theorems 1.1 and 1.2 of \cite{HLNPS1}]\label{thm:stable} For any
$\varepsilon > 0$ there is a $\delta > 0$ and an $N$ such
that for any $n$-vertex graph $G$ with $n>N$ and $d(\overline{K_3},G) \ge p$ and
$|d(K_3,G)-M_p|<\delta$ is $\varepsilon$-close to a graph $H$ on $n$
vertices consisting of a clique and isolated vertices, or
the complement of a graph consisting of a clique and isolated vertices.
Here $M_p := \max \big( (1-p^{1/3})^3 + 3p^{1/3}(1-p^{1/3})^2, (1-q)^{1/3}\big)$
where $q$ is the unique real root of $q^3 + 3q^2(1-q)=p$; that is,
$M_p$ is the largest possible
value of $d(K_3,G)$ given $d(\overline{K_3},G) = p$.
\end{theorem}

From Theorem~\ref{thm:stable} we derive the following lemma.
Note that there are in fact many permutons representing the dimple $(r,r)$
of the feasible region for 123 versus 321, but only two classes
if we consider permutons with isomorphic comparability graphs to be equivalent.
The class that came from the curve $F_1$ has 1\hh2 density $s^2 \sim
.426$, the other $1-s^2 \sim .574$.
(Interestingly, the other end of the $F_1$ curve---represented
uniquely by the identity permuton---had
1\hh2 density 1, while the $F_2$ class ``began'' at 1\hh2 density 0.  Thus,
the 1\hh2 densities crossed on the
way in from the corners of $\cal B$.)

\begin{lemma}\label{lemma:split}
There is a neighborhood of the point $(r,r)$ in the feasible region
for patterns 1\hh2\hh3 and 3\hh2\hh1 within which
no permuton has 1\hh2-density near $\frac12$.
\end{lemma}

\begin{proof} 
Apply Theorem~\ref{thm:stable} with $\varepsilon = .07$ to get $\delta>0$ with the property stated in the theorem.
Let $\delta' =\min(\delta/2,(M_{r - \delta}-r)/2)$, which yields that $|M_p-r|\le\delta/2$ for $p\in [r-\delta',r]$.
So, if $|\rho_{123}(\gamma)-r|\le\delta'\le\delta/2$ and $p\in [r-\delta',r]$, then $|\rho_{123}(\gamma)-M_p|\le\delta$
as required by the hypothesis of Theorem~\ref{thm:stable} (noting that
$\rho_{123}(\gamma)$ is the triangle density of the comparability graph corresponding to $\gamma$).
We conclude that any permuton $\gamma$ such that $(\rho_{123}(\gamma),\rho_{321}(\gamma))$ lies
in the rectangle $[r-\delta',r+\delta'] \times [r-\delta',r]$ has
1\hh2-density within $.07$ of either .426 or .574, thus outside the range $[.496,.504]$.

The symmetric argument gives the same conclusion for $(\rho_{123}(\gamma),\rho_{321}(\gamma))$
in the rectangle $[r-\delta',r] \times [r-\delta',r+\delta']$.
Since there are no permutons $\gamma$ with both $(\rho_{123}(\gamma)$ and $\rho_{321}(\gamma)$ larger than $r$,
the lemma follows.
\end{proof}

\section{Proof of Theorem \protect{\ref{thm:tras}}\label{appendix}}

For completeness, we now give a proof of Theorem~\ref{thm:tras}.
We begin with a simple lemma.

\begin{lemma}\label{lemma:usc}
The function $H:~\Gamma \to \R$ is upper semicontinuous.
\end{lemma}
\begin{proof}Let $\gamma_1, \gamma_2, \dots$ be a sequence of permutons approaching the permuton $\gamma$
(in the $d_\square$-topology); we need to show that $H(\gamma) \ge \lim\sup H(\gamma_n)$.

If $H(\gamma)$ is finite, fix $\eps > 0$ and take $m$ large enough so that $|H(\gamma^m) - H(\gamma)| < \eps$; then
since $H(\gamma^m_n)\ge H(\gamma_n)$ by concavity,
\be
\lim\sup_n H(\gamma_n) \le \lim\sup_n H(\gamma^m_n) = H(\gamma^m) < \eps + H(\gamma)
\ee
and since this holds for any $\eps>0$, the claimed inequality follows.

If $H(\gamma)=-\infty$, fix $t<0$ and take $m$ so large that $H(\gamma^m)<t$.  Then
\be
\lim\sup_n H(\gamma_n) \le \lim\sup_n H(\gamma^m_n) = H(\gamma^m) < t
\ee
for all $t$, so $\lim\sup_n H(\gamma_n^m) \to -\infty$ as desired.
\end{proof}

Let $B(\gamma,\eps) = \{\gamma'|d_\square(\gamma,\gamma') \le \eps\}$ be the (closed) ball in $\Gamma$ of radius $\eps>0$ centered at the permuton $\gamma$,
and let $B_n(\gamma,\eps)$ be the set of permutations $\pi \in S_n$ with $\gamma_\pi \in B(\gamma,\eps)$.

\begin{lemma}\label{lemma:ball}
 For any permuton $\gamma$, $\lim_{\eps \downarrow 0} \lim_{n \to \infty} \frac1n \log (|B_n(\gamma,\eps)|/n!)$ exists and equals $H(\gamma)$.
\end{lemma}

\begin{proof} Suppose $H(\gamma)$ is finite.  It suffices to produce two sets of permutations, $U \subset B_n(\gamma,\eps)$ and $V \supset B_n(\gamma,\eps)$, each
of size
\be
\exp\left(n \log n - n + n(H(\gamma) + o(\eps^0)) + o(n)\right)
\ee
where by $o(\eps^0)$ we mean a function of $\eps$ (depending on $\gamma$) which
approaches 0 as $\eps \to 0$.  (The usual notation here would be $o(1)$; we use
$o(\varepsilon^0)$ here and later to make it clear that the relevant variable is $\varepsilon$
and not, e.g., $n$.)

To define $U$, fix $m > 5/\eps$ so that $|H(\gamma^m) - H(\gamma)| < \eps$ and let $n$ be a multiple of $m$ with $n > m^3/\eps$.
Choose integers $n_{i,j}$, $1 \le i,j \le m$, so that:
\begin{enumerate}
\item $\sum_{i=1}^n n_{i,j} = n/m$ for each $j$;
\item $\sum_{j=1}^n n_{i,j} = n/m$ for each $i$; and
\item $|n_{i,j} - n\gamma(Q_{ij})| < 1$ for every $i,j$.
\end{enumerate}
The existence of such a rounding of the matrix $\{n\gamma(Q_{ij})\}_{i,j}$ is guaranteed by Baranyai's rounding lemma \cite{B}.

Let $U$ be the set of permutations $\pi \in S_n$ with exactly $n_{i,j}$ points in the square $Q_{ij}$,
that is, $|\{i:~ (i/n,\pi(i)/n) \in Q_{ij}\}| = n_{i,j}$, for every $1 \le i,j \le m$.
We show first that $U$ is indeed contained in $B_n(\gamma,\eps)$.  Let $R = [a,b]\times [c,d]$ be a
rectangle in $[0,1]^2$.  $R$ will contain all $Q_{ij}$ for $i_0 < i < i_1$ and $j_0 < j < j_1$ for
suitable $i_0$, $i_1$, $j_0$ and $j_1$, and by construction the $\gamma_\pi$-measure of the union
of those rectangles will differ from its $\gamma$-measure by less than $m^2/n < \eps/m$.  The squares
cut by $R$ are contained in the union of two rows and two columns of width $1/m$, and hence, by
the construction of $\pi$ and the uniformity of the marginals of $\gamma$, cannot contribute more
than $4/m < 4\eps/5$ to the difference in measures.  Thus, finally, $d_\square(\gamma_\pi,\gamma) < \eps/m + 4\eps/5 < \eps$.

Now we must show that $|U|$ is close to the claimed size 
\be \exp\big(n \log n - n - H(\gamma)n\big)\ee  
We construct
$\pi \in U$ in two phases of $m$ steps each.  In step $i$ of Phase I, we decide for each $k$, $(i{-}1)n/m < k \le in/m$,
which of the $m$ $y$-intervals $\pi(k)$ should lie in. There are
\be
{n/m \choose n_{i,1},n_{i,2},\dots,n_{i,m}} = \exp\big((n/m)h_i + o(n/m)\big)
\ee
ways to do this, where $h_i = -\sum_{j=1}^m (n_{i,j}/(n/m)) \log (n_{i,j}/(n/m))$ is the entropy of
the probability distribution $n_{i,\cdot}/(n/m)$. 

Thus, the number of ways to accomplish Phase I is
\begin{eqnarray}
\exp\big(o(n)&+&(n/m)\sum_i h_i\big) =  \exp\big(o(n)-\sum_{i,j} n_{i,j} \log (n_{i,j}/(n/m))\big)\cr
& = & \exp\left(o(n)-\sum_{i,j}n_{i,j} (\log (n_{i,j}/n) + \log m)\right)\cr
& = & \exp\left(o(n) - n\log m - \sum_{i,j}n_{i,j} \log \gamma(Q_{ij})\right)\cr
& = & \exp\left(o(n) - n \log m - n\sum_{i,j} \gamma(Q_{ij}) \log \gamma(Q_{ij})\right)~.
\end{eqnarray}
Recalling that the value taken by the density $g^m$ of $\gamma^m$ on the points of $Q_{ij}$ is $m^2\gamma(Q_{ij})$, we have that
\begin{eqnarray}
H(\gamma^m)& =& \sum_{i,j}\frac{1}{m^2} \big(-m^2 \gamma(Q_{ij}) \log (m^2 \gamma(Q_{ij}))\big)\cr
&=& -\sum_{i,j} \gamma(Q_{ij}) (\log \gamma(Q_{ij}) + 2\log m) \cr
&=& -\sum_{i,j} \gamma(Q_{ij}) (\log \gamma(Q_{ij}) + 2\log m) \cr
&=& -2\log m - \sum_{i,j} \gamma(Q_{ij}) \log \gamma(Q_{ij})~.
\end{eqnarray}
Therefore we can rewrite the number of ways to do Phase I as  
\be\exp\big(n \log m + nH(\gamma^m) + o(n)\big).\ee

In Phase II we choose a permutation $\pi_j \in S_{n/m}$ for each $j$, $1 \le j \le m$, and order the $y$-coordinates of the $n/m$
points (taken left to right) in row $j$ according to $\pi_j$.  Together with Phase I this determines $\pi$ uniquely, and the number
of ways to accomplish Phase II is
\begin{eqnarray}
(n/m)!^m &=& \left(\exp\left(\frac{n}{m} \log \frac{n}{m} -
    \frac{n}{m} + o(n/m)\right)\right)^m \cr
&=& \exp\big(n \log n - n - n \log m + o(n)\big)
\end{eqnarray}
so that in total,
\begin{eqnarray}
|U| &\ge&  \exp\big(n \log m + nH(\gamma^m) + o(n)\big)\exp\big(n \log n - n - n \log m + o(n)\big)\cr
&=& \exp\big(n \log n - n + nH(\gamma^m) + o(n)\big)
\end{eqnarray}
which, since $|H(\gamma)-H(\gamma^m)| < \eps$, does the job.

We now proceed to the other bound, which involves similar calculations in a somewhat different context.  To define the required set
$V \supset B_n(\gamma,\eps)$ of permutations we must allow a wide range for the number of points of $\pi$ that fall in each square $Q_{ij}$---
wide enough so that a violation causes $Q_{ij}$ itself to witness $d_\square(\gamma_\pi,\gamma) > \eps$, thus guaranteeing that if $\pi \not\in V$
then $\pi \not\in B_n(\gamma,\eps)$.

To do this we take $m$ large, $\eps < 1/m^4$, and $n > 1/\eps^2$.  We define $V$ to be the set
of permutations $\pi \in S_n$ for which the number of points $(k/n,\pi(k)/n)$ falling in $Q_{ij}$ lies in the range
$[n(\gamma(Q_{ij})-\sqrt\eps),n(\gamma(Q_{ij})+\sqrt\eps)]$.  Then, as promised, if $\pi \not\in V$ we have a rectangle $R = Q_{ij}$ with
$|\gamma(R) - \gamma_\pi (R)| > \sqrt\eps/m^2 > \eps$.

It remains only to bound $|V|$.  Here a preliminary phase is needed in which the exact count of points in each square $Q_{ij}$ is determined; since
the range for each $n_{i,j}$ is of size $2n\sqrt\eps$, there are at most $\left(2n\sqrt\eps\right)^{m^2} = \exp\big(m^2\log (2n\sqrt\eps)\big)$ ways to do this,
a negligible factor since $m^2\log(n\sqrt\eps) = o(n)$.  For Phase I we must assume the $n_{i,j}$ are chosen to maximize each $h_i$ but since the
entropy function $h$ is continuous, the penalty shrinks with $\eps$.  Counting as before, we deduce that here the number of ways to accomplish Phase I is bounded by
\begin{eqnarray}
& & \exp\big(n \log m + n(H(\gamma^m) + o(\eps^0)) + o(n)\big)\nonumber\\
& & = \exp\big(n \log m + n(H(\gamma) + o(\eps^0)) + o(n)\big).
\end{eqnarray}
The computation for Phase II is exactly as before and the conclusion is that
\begin{eqnarray}
|V| &\le& \exp\big(n \log m - n + n(H(\gamma) + o(\eps^0)) + o(n)\big) \cr
&&\hskip1cm\times\exp\big(n \log n - n - n \log m + o(n)\big)\cr
&=& \exp\big(n \log n - n + nH(\gamma) + o(n)\big)
\end{eqnarray}
proving the lemma in the case where $H(\gamma) > -\infty$.

If $H(\gamma) > -\infty$, we need only the upper bound provided by the set $V$.  Fix $t < 0$ with the idea of showing that 
$\frac{1}{n}\log\frac{|B_n(\gamma,\eps_\gamma)|}{n!} < t$.  Define $V$ as above, first insuring that $m$ is large enough so that
$H(\gamma^m) < t{-}1$.  Then the number of ways to accomplish Phase I is bounded by
\be
\exp\big(n \log m +n(H(\gamma^m) + o(\eps^0)) + o(n)\big)\\
< \exp\big(n \log m + n(t{-}1 + o(\eps^0)) + o(n)\big)
\ee
and consequently $|V|$ is bounded above by
\be
\exp\big(n \log n - n + n(t{-}1) + o(n)\big) < \exp\big(n \log n - n + nt\big)~.
\ee
\end{proof}

We are finally in a position to prove Theorem~\ref{thm:tras}.  If our set $\Lambda$ of permutons is closed, then, since $\Gamma$ is compact, so is $\Lambda$.
Let $\delta>0$ with the idea of showing that
\be
\lim_{n \to \infty}\frac{1}{n}\log\frac{|\Lambda_n|}{n!} \le H(\mu) + \delta
\ee
for some $\mu \in \Lambda$.  If not, for each $\gamma \in \Lambda$ we may, on account of Lemma~\ref{lemma:ball}, choose $\eps_\gamma$ and $n_\gamma$ so that
$\frac{1}{n}\log\frac{|B_n(\gamma,\eps_\gamma)|}{n!} < H(\gamma)+\delta/2$ for all $n \ge n_\gamma$.  Since a finite number of these balls cover $\Lambda$, we
have too few permutations in $\Lambda_n$ for large enough $n$, and a contradiction has been reached.

If $\Lambda$ is open, we again let $\delta>0$, this time with the idea of showing that 
\be
\lim_{n \to \infty}\frac{1}{n}\log\frac{|\Lambda_n|}{n!} \ge H(\mu) - \delta~.
\ee
To do this we find a permuton $\mu \in \Lambda$ with
\be
H(\mu) > \sup_{\gamma \in \Lambda}H(\gamma) - \delta/2~,
\ee
and choose $\eps > 0$ and $n_0$ so that $B_n(\mu,\eps) \subset \Lambda$ and
$\frac{1}{n}\log\left(\frac{|B_n(\mu,\eps)|}{n!}\right) > H(\mu)-\delta/2$ for $n \ge n_0$.

\medskip

This concludes the proof of Theorem~\ref{thm:tras}.

\section*{Acknowledgments}

The second author would like to thank Roman Glebov, Andrzej Grzesik
and Jan Volec for discussions on pattern densities in permutons, and
the third author would like to thank Sumit Mukherjee for pointing out
the papers \cite{Mu} and \cite{Tr}.

The authors gratefully acknowledge the support of the National Science Foundation's
Institute for Computational and Experimental Research in Mathematics, at Brown University,
where this research was conducted.
This work was also partially supported by NSF grants DMS-1208191,
DMS-1509088 and DMS-0901475, the Simons Foundation award no.~327929, 
and by the European Research Council under thev
European Union's Seventh Framework Program (FP7/2007-2013)/ERC grant
agreement no.~259385.

\end{document}